\documentclass[11pt]{article}
\usepackage{amsmath,amssymb}

\def\N{\mathbb{N}}

\def\R{\mathbb{R}}

\def\C{C^{\infty}(M)}

\newtheorem{definition}{Definition}[section]

\newtheorem{proposition}[definition]{Proposition}
\newtheorem{theorem}[definition]{Theorem}
\newtheorem{remark}[definition]{Remark}
\newtheorem{corol}[definition]{Corollary}

\newenvironment{proof}{\noindent{\bf Proof.}}{\hfill $\blacklozenge$}
\newenvironment{sk-proof}{\noindent{\bf Sketch of the proof.}}{\hfill $\blacklozenge$}

\begin{document}

\title{Poisson brackets with prescribed family of functions in involution}

\author{Fani Petalidou
\\
\\
\emph{Department of Mathematics}\\
\emph{Aristotle University of Thessaloniki} \\
\emph{54124 Thessaloniki, Greece} \\
\\
\noindent\emph{E-mail: petalido@math.auth.gr}}

\date{}
\maketitle

\vskip 1 cm

\begin{abstract}
It is well known that functions in involution with respect to Poisson brackets have a privileged role in the theory of completely integrable systems. Finding functionally independent functions in involution with a given function $h$ on a Poisson manifold is a fundamental problem of this theory and is very useful for the explicit integration of the equations of motion defined by $h$. In this paper, we present our results on the study of the inverse, so to speak, problem. By developing a technique analogous to that presented in \cite{dam-pet} for the establishment of Poisson brackets with prescribed Casimir invariants, we construct an algorithm which yields Poisson brackets having a given family of functions in involution. Our approach allows us to deal with bi-Hamiltonian structures constructively and therefore allows us to also deal with the completely integrable systems that arise in such a framework.
\end{abstract}

\vspace{5mm} \noindent {\bf{Keywords: }}{Poisson bracket, Casimir function, functions in involution, Poisson pencil, bi-Hamiltonian hierarchy, Casimir of a Poisson pencil.}

\vspace{3mm} \noindent {\bf{MSC (2010):}} {\textbf{53D17, 37K10}}.

\section{Introduction}
A \emph{Poisson bracket} $\{\cdot, \cdot\}$ on the space $\C$ of smooth functions on a smooth manifold $M$ is an internal bilinear and skew-symmetric composition law that verifies Jacobi's identity and Leibniz's rule. On the space of smooth functions on the phase space of a mechanical system, this concept was introduced by S. D. Poisson \cite{poi} in 1809 in order to study the problem of variation of constants in the framework of analytical mechanics. Later, it was used by Jacobi in his study on the problem of integration of partial differential equations \cite{jac}, but, the study of its geometry began by the works of S. Lie \cite{lie} who generalized this notion to manifolds of arbitrary dimension. The increased interest in this subject during the 19th century was originally motivated by the important role of Poisson brackets in Hamiltonian dynamics. After a long period of inactivity, it has been revived in the last 40 years due to the publications of A. Lichnerowicz \cite{lch1}, A. Kirillov \cite{kir} and A. Weinstein \cite{wei}. Now, \emph{Poisson geometry}, as most people call the branch of Differential Geometry that studies Poisson manifolds $(M, \{\cdot,\cdot\})$, i.e., smooth manifolds $M$ whose space of smooth functions $\C$ is endowed with a Poisson bracket $\{\cdot,\cdot\}$, has become a major and active branch stimulated by connections with several disciplines of Mathematics and Mathematical Physics. For a detailed presentation of the scientific and historical development of this subject we refer the book of Y. Kosmann-Schwarzbach \cite{yks-Poisson}.

On a Poisson manifold $(M, \{\cdot,\cdot\})$, we say that two elements $f, g$ of $\C$ are in \emph{involution} if they pairwise commute, i.e.,
\begin{equation*}
\{f,g\}=0.
\end{equation*}
They form an important class of functions on $(M, \{\cdot,\cdot\})$ and they have a dominant role in the integrability problem of Hamiltonian systems. It is well known, a Poisson bracket defines an operator which associates to every smooth function $h$ on $M$ a vector field $X_h$, called \emph{Hamiltonian vector field of $h$ with respect to $\{\cdot,\cdot\}$}. The integrability problem of such a dynamical system investigates the existence of sufficiently many first integrals (constants of motion) which render the integration of differential equations possible. An important property of these vector fields is that their constants of motion are in involution with $h$. Poisson's theorem, which states that the Poisson bracket of two constants of motion of a Hamiltonian system is also a constant of motion of the system (i.e., in other words, the Poisson bracket of two functions in involution with $h$ is a function in involution with $h$), was in Poisson's time a fundamental result that suggested a method of construction of first integrals of $X_h$. However, the fundamental role of functions in involution in integrability theory was pointed out by Liouville's theorem \cite{liou} and its generalization, which was established by V. I. Arnold in \cite{arn} and gave to it a geometrical interpretation (see also \cite{lib-marle}, \cite{per}). The exact formulation of this theorem, known as \emph{Arnold--Liouville Theorem}, is: \emph{Let $f_1,\ldots,f_n$ be $n$ smooth functions in involution on a canonical $2n$-dimensional Poisson manifold $(M,\{\cdot,\cdot\})$. Consider a level set
$\mathcal{M}_c = \{x\in M \, / \, f_1(x)=c_1, \ldots, f_n(x)=c_n\}$ of the functions $f_1,\ldots,f_n$, $c=(c_1,\ldots,c_n)\in \R^n$. Assume that the $n$ functions $f_1,\ldots,f_n$ are functionally independent on $\mathcal{M}_c$. Then
\begin{enumerate}
\item
$\mathcal{M}_c$ is a smooth submanifold of $M$, invariant under the flow of each Hamiltonian vector field $X_{h}$, with $h=f_i$, $i=1,\ldots,n$.
\item
If $\mathcal{M}_c$ is compact and connected, then it is diffeomorphic to the $n$-dimensional torus.
\item
The flow of $X_h$ determines a conditionally periodic motion on $\mathcal{M}_c$.
\item
The solutions of Hamilton's equations $\dot{x}=\{h,x\}$ lying in $\mathcal{M}_c$ can be obtained by quadratures.
\end{enumerate}}
\noindent
We note that in the context of degenerate Poisson structures (of constant rank), Arnold-Liouville Theorem still holds, under some minor adaptations \cite{adler-vanMoer-Pol}.

The above results justify the key role of functions in involution in the study of the integrability of Hamiltonian systems and the interest for the development of methods of constructing functions in involution. Besides Poisson's theorem mentioned above, some other such methods are \cite{CaAnPo}:
\begin{itemize}
\item[-]
\emph{Lax's equations:} If a vector field $X$ on $M$ satisfies Lax's equation
\begin{equation*}
\dot{L}(t)=[L,P],
\end{equation*}
where $(L,P)$ is a Lax pair of matrices, then the functions $\mathrm{Trace}(L^k)$, $k\in \N$, are constants of motion of $X$. The same is true for any coefficient of the characteristic polynomial $\det(L-\lambda I)$ of $L$, since the latter is a polynomial function of $\mathrm{Trace}(L^k)$, $k\in \N$. In particular, if $M$ is a Poisson manifold and $X$ is a Hamiltonian vector field with Hamiltonian function $h$, then $h$ is in involution with each $\mathrm{Trace}(L^k)$, $k\in \N$, and, consequently, with each coefficient of $\det(L-\lambda I)$.
\item[-]
\emph{Thimm's method:} Let $\mathfrak{g}$ be a finite dimensional Lie algebra and let $\mathfrak{g}_1$, $\mathfrak{g}_2$ be two Lie subalgebras of $\mathfrak{g}$ such that $[\mathfrak{g}_1, \mathfrak{g}_2]\subset \mathfrak{g}_2$. We denote by $i_k : \mathfrak{g}_k \hookrightarrow \mathfrak{g}$, $k=1,2$, the inclusions maps and by $i_k^\ast : \mathfrak{g}^\ast \to \mathfrak{g}_k^\ast$ their transpose maps. Then, for any function $f_1$ on $\mathfrak{g}_1^\ast$ and for any $Ad^\ast$-invariant function $f_2$ on $\mathfrak{g}_2^\ast$, the functions $f_1\circ i_1^\ast, f_2\circ i_2^\ast \in C^\infty(\mathfrak{g}^\ast)$ are in involution with respect to Lie-Poisson bracket on $\mathfrak{g}^\ast$.
\item[-]
\emph{Bi-Hamiltonian vector fields:} If a dynamical system $X$ on $M$ can be written in Hamiltonian form in two different compatible ways:
\begin{equation*}
X=\{h_1,\cdot\}_0 = \{h_0,\cdot\}_1, \quad \quad h_0, h_1\in \C,
\end{equation*}
then $h_0$ and $h_1$ are in involution with respect to both Poisson brackets $\{\cdot,\cdot\}_0$ and $\{\cdot,\cdot\}_1$ and to any linear combination of these. Hence, in this case, we can say that $h_0$ and $h_1$ are in \emph{bi-involution}.
\end{itemize}

In this paper we deal with the inverse, so to speak, problem to the one of constructing functions in involution on Poisson manifolds which can be formulated as follows:

\vspace{2mm}
\noindent
{\emph{For a preassigned family $\mathfrak{F}=(f_1,\ldots,f_{r+k})$ of $r+k$ smooth functions on a $2r+k$--dimensional smooth manifold $M$, functionally independent almost everywhere, determine the Poisson brackets $\{\cdot,\cdot\}$ of rank at most $2r$ on $\C$ with respect to which the elements of $\mathfrak{F}$ are in involution.}}

\vspace{2mm}
\noindent
We note that we chose the number $r+k$ as cardinal number of the set $\mathfrak{F}$, because it is a lower bound on the number of functions in involution on $(M,\{\cdot,\cdot\})$ which are enough to completely integrate the Hamiltonian systems on this manifold. This problem, besides being interesting in its own right, it is also closely connected to the theory of \emph{Inverse problems for Differential Equations} \cite{Lli-Ram}. In particular, it is related to the problem of determining, for a given submanifold $\mathcal{M}$ of a manifold $M$, the differential systems whose local flow leaves $\mathcal{M}$ invariant. In our case, the smooth map $\mathfrak{F}=(f_1,\ldots,f_{r+k}) : M \to \R^{r+k}$ (at this point, we denote also by $\mathfrak{F}$ the smooth map defined by the family of functions $\mathfrak{F}$) is a submersion on the open and dense subset $\mathcal{U}$ of $M$ where the components of $\mathfrak{F}$ are functionally independent. Therefore, for any regular value $c=(c_1,\ldots,c_{r+k}) \in \R^{r+k}$ of $\mathfrak{F}$, the set $\mathcal{M}_c= \mathfrak{F}^{-1}(c)= \{x\in \mathcal{U} \,/\, f_1(x) = c_1, \, \ldots, \,  f_{r+k}(x) = c_{r+k}\}$ is a submanifold of $\mathcal{U}$. Hence, by determining the Poisson brackets $\{\cdot,\cdot\}$ on $C^\infty(\mathcal{U})$ for which $\mathfrak{F}$ (as family of functions) is involutive, we can determine differential vector fields on $M$ whose local flow leaves the submanifold $\mathcal{M}_c$ invariant. They will be the Hamiltonian vector fields $\{f_i,\cdot\}$. Our approach for this study rests, as we will explain in Section \ref{section-central}, on the properties of the Casimir invariants of a Poisson pencil, i.e., of a $1$-parameter family of compatible Poisson brackets $\{\cdot , \cdot\}_{(\lambda)}$, $\lambda \in \R\cup \{\infty\}$,
in the sense that any linear combination of elements of this family is also a Poisson bracket, and on the results of \cite{dam-pet} concerning the Poisson brackets with prescribed family of Casimirs. Firstly, we investigate the case where $M$ is of even dimension $2r+2l$, i.e., $k=2l$, and we prove that the brackets with the required properties are given by the formula \begin{equation*}
\{\cdot , \cdot\}_{(\lambda)} \Omega = -\frac{1}{F(\lambda)}d\cdot \wedge\, d\cdot \wedge\big(\sigma_{(\lambda)} + \frac{g_{(\lambda)}}{r-1}\omega\big) \wedge \frac{\omega^{r-2}}{(r-2)!}\wedge dF^1(\lambda)\wedge\ldots \wedge dF^k(\lambda),
\end{equation*}
where $\omega$ is an appropriate (almost) symplectic form on $M$, $\Omega$ is the volume form on $M$ defined by $\omega$, $\sigma_{(\lambda)}$ is a $2$-form on $M$ depended on the parameter $\lambda$ and satisfying some special requirements (see, formul{\ae} (\ref{sigma}) and (\ref{recursion - sigma - X})), $g_{(\lambda)} = i_{\Lambda}\sigma_{(\lambda)}$ ($\Lambda$ being the (almost) Poisson tensor field on $M$ associated to $\omega$), $F^1(\lambda), \ldots, F^k(\lambda)$ are polynomials of the (real) parameter $\lambda$ with coefficients from $\mathfrak{F}$, and $F(\lambda)$ is the polynomial, non-vanishing almost everywhere on $M$, defined by the relation $F(\lambda)^2=\det \big(\{F^i(\lambda), F^j(\lambda)\}_{\omega}\big)$ (with $\{\cdot,\cdot\}_{\omega}$ being the (Poisson) bracket on $\C$ determined by $\omega$). Thereafter, a similar formula is established in the case where $M$ is of odd dimension $2r+2l+1$, i.e., $k=2l+1$, and endowed with an (almost) co-symplectic structure $(\vartheta, \Theta)$ which defines the volume form $\Omega = \vartheta\wedge \Theta^{r+l}$. Then, we show that the demanded brackets are written as \begin{equation*}
\{\cdot , \cdot\}_{(\lambda)} \Omega = -\frac{1}{F(\lambda)}d\cdot \wedge \, d\cdot \wedge\big(\sigma_{(\lambda)} + \frac{g_{(\lambda)}}{r-1}\Theta\big) \wedge \frac{\Theta^{r-2}}{(r-2)!}\wedge dF^1(\lambda)\wedge\ldots \wedge dF^k(\lambda),
\end{equation*}
where $\sigma_{(\lambda)}$ is also a $2$-form on $M$ depended on $\lambda$ and having some particular properties, $g_{(\lambda)} = i_{\Lambda}\sigma_{(\lambda)}$ ($\Lambda$ being the bi-vector field on $M$ defined by $(\vartheta, \Theta)$) and $F^1(\lambda), \ldots, F^k(\lambda)$ are, as in the previous case, polynomials of the (real) parameter $\lambda$ with coefficients from $\mathfrak{F}$.

The paper is organized as follows. In Section \ref{preliminaries} we review the preliminary notions and results that concern Poisson brackets, Poisson pencils and Casimir invariants of a Poisson pencil and that are used in the study which follows. A detailed presentation of our strategy for solving the problem, of the proofs of the main theorems (Theorem \ref{central-theorem} and Theorem \ref{theorem-odd}) and of their parallel results are given in Section \ref{section-central}. Finally, in Section \ref{examples}, we illustrate our theory by some examples that concern the Lagrange's top and the Toda lattice.

\section{Preliminaries}\label{preliminaries}
We start this section by fixing our notation and by recalling briefly some basic notions and results on the theory of Poisson manifolds, which are the natural setting for the study of Hamiltonian dynamical systems. The reference books on this subject are \cite{lib-marle}, \cite{vai-b}, \cite{duf-zung}, \cite{CaAnPo}. The facts that concern bi-Hamiltonian structures can be found in a series of articles of F. Magri and his collaborators. For instance, we cite the papers \cite{magri}, \cite{casati}. In the following, $M$ denotes a finite dimensional real smooth manifold, $TM$ and $T^\ast M$ its tangent and cotangent bundle, respectively, and $C^{\infty}(M)$ the space of smooth functions on $M$. For each $p\in \mathbb{Z}$, we write $\mathcal{V}^p(M)$ and $\Omega^p(M)$ for the space of smooth sections of $\bigwedge^p TM$ and $\bigwedge^p T^\ast M$, respectively. By convention, we set $\mathcal{V}^p(M) = \Omega^p(M) = \{0\}$, for $p<0$, $\mathcal{V}^0(M) = \Omega^0(M) = C^{\infty}(M)$, and, taking into account the skew-symmetry, we have $\mathcal{V}^p(M) = \Omega^p(M) = \{0\}$, for $p>\dim M$. Finally, we set $\mathcal{V}(M)=\oplus_{p\in \mathbb{Z}}\mathcal{V}^p(M)$ and $\Omega(M) = \oplus_{p\in \mathbb{Z}}\Omega^p(M)$.

\subsection{Background on Poisson manifolds}
A \emph{Poisson structure} on a $m$-dimensional smooth manifold $M$ is defined by a bilinear, skew-symmetric map
\begin{equation*}
\{ \cdot, \cdot\} : \C \times \C \to \C
\end{equation*}
that acts as a derivation on itself and on the usual product $"\cdot"$ of functions. The latter means that $\{\cdot,\cdot\}$ verifies, for any $f,g,h \in \C$, the Jacobi identity
\begin{equation*}
\{f,\{g,h\}\} = \{\{f,g\},h\}+\{g,\{f,h\}\}
\end{equation*}
and the Leibniz's rule
\begin{equation*}
\{f, g \cdot h\} = \{f,g\}\cdot h + g\cdot \{f,h\}.
\end{equation*}
So, $\{\cdot,\cdot\}$ defines on $\C$ a Lie algebra structure which is compatible with the product $"\cdot"$. Such a bracket is called \emph{Poisson bracket} and it can be viewed as an operator which associates to every smooth function on $M$ a vector field on $M$. In fact, since it is a derivation on the algebra $(\C, \cdot)$, by fixing the first argument and keeping the other one free, we obtain, for any $f\in \C$, a vector field $X_f = \{f, \cdot\}\in \mathcal{V}^1(M)$ which is called the \emph{Hamiltonian vector field associated to $f$ with respect to $\{\cdot, \cdot\}$}.

By virtue of the above properties, A. Lichnerowicz proved in \cite{lch1} that $\{\cdot,\cdot\}$ gives rise to a bivector field $\Pi$ on $M$, $\Pi: M \to \bigwedge^2TM$, such that
\begin{equation*}
\Pi(df,dg)=\{f,g\} \quad \mathrm{and} \quad [\Pi,\Pi]=0,
\end{equation*}
where $[\cdot,\cdot]$ denotes the Schouten bracket. Conversely, any bivector field $\Pi \in \mathcal{V}^2(M)$ that verifies the last condition defines a Poisson bracket $\{\cdot,\cdot\}$ given, for any $f,g\in \C$, by $\{f,g\}=\Pi(df,dg)$. In this case, $\Pi$ is called a \emph{Poisson tensor} and the pair $(M,\Pi)$ a \emph{Poisson manifold}.

Moreover, $\Pi$, as any bivector field, defines a natural homomorphism $\Pi^\#: \Omega^1(M)\to \mathcal{V}^1(M)$ which maps each element $\alpha \in \Omega^1(M)$ to a unique vector field $\Pi^\#(\alpha)$ such that, for any $\beta\in \Omega^1(M)$,
\begin{equation*}
\langle \beta, \Pi^\#(\alpha)\rangle = \Pi(\alpha,\beta).
\end{equation*}
We remark that $\{f,\cdot\}=\Pi(df,\cdot)=\Pi^\#(df)$. Thus, a Hamiltonian vector field $X_f$, $f\in \C$, on $(M,\Pi)$ can be viewed as the image by $\Pi^\#$ of the exact $1$-form $df$ on $M$. The map $\Pi^\#$ can be extended to a homomorphism, also denoted by $\Pi^\#$, from $\Omega^p(M)$ to $\mathcal{V}^p(M)$, $p\in \N$, by setting, for any $f\in \C$, $\Pi^\#(f)=f$, and, for any $\zeta \in \Omega^p(M)$ and $\alpha_1,\ldots, \alpha_p \in \Omega^1(M)$,
\begin{equation}\label{def-extension P}
\Pi^\#(\zeta)(\alpha_1,\ldots,\alpha_p)=(-1)^p\zeta(\Pi^\#(\alpha_1),\ldots,\Pi^\#(\alpha_p)).
\end{equation}
Thus, $\Pi^\#(\zeta \wedge \eta) = \Pi^\#(\zeta)\wedge \Pi^\#(\eta)$, for all $\eta \in \Omega(M)$. In particular, when $\Pi$ is nondege\-nerate, $\Pi^\# : \Omega^p(M) \to \mathcal{V}^p(M)$, $p\in \N$, is an isomorphism and any bivector field $\Pi'$ on $(M, \Pi)$ can be viewed as the image $\Pi^\#(\sigma)$ of a $2$-form $\sigma$ on $M$ by $\Pi^\#$. In \cite{dam-pet}, we proved:
\begin{proposition}\label{prop-Pi-sigma}
The bivector field $\Pi' = \Pi^\#(\sigma)$ defines a Poisson structure on $(M, \Pi)$ if and only if
\begin{equation*}
\delta (\sigma\wedge\sigma) = 2\sigma\wedge\delta(\sigma),
\end{equation*}
where $\delta = \ast d \ast$ is the codifferential operator of degree $-1$ defined on $\Omega(M)$ by the isomorphism $\ast : \Omega^p(M) \to \Omega^{2n-p}(M)$ which maps any $p$-form $\varphi$ on $M$ to the $(2n-p)$-form $\ast \varphi = i_{\Pi^\#(\varphi)}\displaystyle{\frac{\omega^n}{n!}}$, $\omega$ being the almost symplectic form on $M$ defined by $\Pi$.
\end{proposition}

In the same paper \cite{lch1}, A. Lichnerowicz demonstrated another fundamental and beautiful property of $\Pi^\#$. Its image $\mathrm{Im}(\Pi^\#)$ defines a generalized foliation $\mathcal{S}$ of $M$, called the \emph{symplectic foliation of $(M,\Pi)$}. The leaves of $\mathcal{S}$ are symplectic immersed submanifolds of $(M,\Pi)$ which, locally, can be viewed as the common level sets of $m-2r$ ($2r$ being the rank of the skew-symmetric matrix associated to the Poisson tensor $\Pi$) functionally independent smooth functions on $M$ known as \emph{Casimir invariants of $\Pi$}. They are characterized by the solution of the system of partial differential equations $\Pi^\#(dC)=0$ and they can be considered as the functions $C$ which generate the null Hamiltonian vector field $X_C = \Pi^\#(dC)=0$, or, equivalently, as the elements of the center of the Lie algebra $(\C, \{\cdot,\cdot\})$:
\begin{equation*}
\{C,\cdot\}=0.
\end{equation*}
As such, they are conserved quantities of any Hamiltonian system $X_h$ of $(M,\Pi)$, i.e.,
\begin{equation*}
X_h(C)=\langle dC, X_h \rangle = \langle dC, \Pi^\#(dh)\rangle = \{h,C\}=0.
\end{equation*}
We thus get that they are in involution with any $h\in \C$ and therefore play a dominant role in the study of integrability and reduction of the order of Hamiltonian systems, in the Energy-Casimir method of determining stability of equilibrium points of said systems \cite{jm-mec}, and in a variety of other problems.

\subsection{Poisson brackets with given Casimirs}\label{Poisson given Casimirs}
In \cite{dam-pet} we considered the problem of constructing Poisson brackets on smooth manifolds $M$ with prescribed family of Casimir functions. Our interest for this question was generated by the following facts:
\begin{itemize}
\item[-]
For any smooth function $f$ on $\R^3$ the bracket
\begin{equation*}
\{x,y\}\Omega = dx\wedge dy \wedge df, \quad \{x,z\}\Omega = dx \wedge dz \wedge df, \quad \{y,z\}\Omega = dy\wedge dz \wedge df,
\end{equation*}
where $\Omega = dx\wedge dy\wedge dz$ is the standard volume form on $\R^3$, is Poisson and admits $f$ as Casimir.
\item[-]
If $C_1,\ldots,C_l$ are $l$ functionally independent smooth functions on $\R^{l+2}$ and $\Omega$ a volume element on $\R^{l+2}$, then the formula
\begin{equation}\label{br-Casimir-2}
\{g,h\}\Omega = fdg\wedge dh\wedge dC_1\wedge \ldots \wedge dC_l, \quad \quad g,h \in C^\infty(\R^{l+2}),
\end{equation}
defines a Poisson bracket of rank $2$ on $\R^{l+2}$ with $C_1,\ldots,C_l$ as Casimirs \cite{Grab93}.
\end{itemize}
The above type of Poisson bracket is called \emph{Jacobian Poisson bracket} because $\{g,h\}$ is equal, up to a coefficient function $f$, with the determinant of the usual Jacobian matrix of $(g,h,C_1,\ldots,C_l)$. It has appeared in the theory of transverse Poisson structures to subregular nilpotent orbits of $\mathfrak{gl}(n, \mathbb{C})$, $n\leq 7$, \cite{Dam89}, \cite{Dam96}, and of any semi-simple Lie algebra \cite{Dam07}, and, also, in the theory of polynomial Poisson algebras with some regularity conditions \cite{or}. Hence, our aim was to extend, if possible, formula (\ref{br-Casimir-2}) in the more general case of higher rank Poisson brackets and
study its applications. Firstly, we studied this problem in the case where $M$ is even-dimensional ($\dim M =2n$) endowed with a suitable (almost) symplectic form $\omega$. Then, we proved (see, Theorem 3.3 in \cite{dam-pet}) that, for given $2n-2r$ smooth functions $C_1, \ldots, C_{2n-2r}$ on $M$, functionally independent almost everywhere, and for any $2$-form $\sigma$ on $M$ satisfying certain special requirements, the bracket
\begin{equation*}\label{bracket-even}
\{h_1,h_2\} \Omega =  -\frac{1}{f} dh_1 \wedge dh_2 \wedge (\sigma + \frac{g}{r-1}\omega) \wedge \frac{\omega^{r-2}}{(r-2)!}\wedge dC_1\wedge\ldots \wedge dC_{2n-2r}, \quad h_1, h_2\in \C,
\end{equation*}
is a Poisson bracket of maximal rank $2r$ having the given functions as Casimirs. Here, $\Omega=\displaystyle{\frac{\omega^n}{n!}}$ is a volume element on $M$, $f$ satisfies $f^2 = \det \big(\{f_i,f_j\}_{\omega}\big)\neq 0$ ($\{\cdot,\cdot\}_{\omega}$ being the bracket defined by $\omega$ on $\C$) and $g = i_{\Lambda}\sigma$ ($\Lambda$ being the bivector field on $M$ associated to $\omega$). We proceeded by considering the case $\dim M = 2n+1$ and we established a similar formula for the Poisson brackets on $\C$ of rank at most $2r$ with Casimir invariants a given family $(C_1,\ldots,C_{2n+1-2r})$ of $2n+1-2r$ smooth functions on $M$. More precisely, we proved that such a bracket can be written as
\begin{equation*}
\{h_1,h_2\} \Omega =  -\frac{1}{f} dh_1 \wedge dh_2 \wedge (\sigma + \frac{g}{r-1}\Theta) \wedge \frac{\Theta^{r-2}}{(r-2)!}\wedge dC_1\wedge\ldots \wedge dC_{2n+1-2r}.
\end{equation*}
In the above formula, $(\vartheta,\Theta)$ is a suitable (almost) cosymplectic structure on $M$, $\Omega = \vartheta \wedge \displaystyle{\frac{\Theta^n}{n!}}$ is the corresponding volume form, $f$ an element of $\C$, $\sigma$ a $2$-form on $M$ satisfying certain  particular conditions (see, Theorem 3.7 in \cite{dam-pet}), and $g=i_{\Lambda}\sigma$ ($\Lambda$ being the bivector field on $M$ associated to $(\vartheta,\Theta)$).

\subsection{Casimir invariants of a Poisson pencil}\label{Casimir of pencil}
A \emph{bi-Hamiltonian} structure on a smooth manifold $M$ is defined by a pair $(\Pi_0,\Pi_1)$ of compatible Poisson structures on $M$ in the sense that $\Pi_0+\Pi_1$ is still a Poisson structure. The last condition happens if and only if the Schouten bracket $[\Pi_0,\Pi_1]$ vanishes identically on $M$. Then, any linear combination of $\Pi_0$ and $\Pi_1$ produces another Poisson structure on $M$ and any pair of structures of type $\lambda_0\Pi_0 + \lambda_1\Pi_1$, $\lambda_0,\lambda_1 \in \R$, is also a pair of compatible Poisson structures on $M$.

We consider a bi-Hamiltonian manifold $(M,\Pi_0,\Pi_1)$ and we denote by $\{\cdot,\cdot\}_i$ the Poisson bracket on $\C$ defined by $\Pi_i$, $i=0,1$. For any $\lambda \in \R\cup\{\infty\}$, we set
$\Pi_{(\lambda)} = \Pi_1-\lambda \Pi_0$. It is a Poisson tensor with corresponding Poisson bracket
\begin{equation}\label{br-pencil}
\{\cdot,\cdot\}_{(\lambda)} = \{\cdot,\cdot\}_1-\lambda\{\cdot,\cdot\}_0.
\end{equation}
The $1$-parameter family of Poisson tensors $(\Pi_{(\lambda)})_{\lambda \in \R\cup\{\infty\}}$ on $M$ is referred as \emph{Poisson pencil} and the family (\ref{br-pencil}) as \emph{pencil of Poisson brackets}. Analogously to the definition of Casimir function of a Poisson bracket, we define the notion of \emph{Casimir invariant of a Poisson pencil} as a function depended on the parameter $\lambda$ which commutes with any other function on $M$ with respect to each bracket of the family (\ref{br-pencil}).

A \emph{bi-Hamiltonian} hierarchy on $(M,\Pi_0,\Pi_1)$ is a sequence $(h_k)_{k\in \N}$ of smooth functions on $M$ fulfilling the Lenard-Magri recursion relations:
\begin{equation}\label{Lenard-Magri-1}
\{\cdot,h_{k+1}\}_0 = \{\cdot,h_k\}_1, \quad \quad k\in \N.
\end{equation}
In terms of Poisson tensors, equations (\ref{Lenard-Magri-1}) are written as
\begin{equation*}
\Pi_0^\#(dh_{k+1})=\Pi_1^\#(dh_k), \quad \quad k\in \N,
\end{equation*}
which means that a bi-Hamiltonian hierarchy gives rise to an infinite sequence of bi-Hamiltonian vector fields.

\begin{proposition}\label{hierarchy 1}
The functions of a bi-Hamiltonian hierarchy $(h_k)_{k\in \N}$ on $(M,\Pi_0,\Pi_1)$ are in involution with respect to any Poisson bracket of the pencil (\ref{br-pencil}).
\end{proposition}
\begin{proof}
Indeed, for $k>j$, we set $d=k-j$ and we have
\begin{eqnarray*}
\{h_j, h_k\}_0 & = & \{h_j, h_{k-1}\}_1 = \{h_{j+1}, h_{k-1}\}_0 = \ldots = \{h_{j+d},h_{k-d}\}_0\\
  & = & \{h_k, h_j\}_0.
\end{eqnarray*}
So, $\{h_j, h_k\}_0 = 0$, for all $j, k \in \N$. Similarly, we get $\{h_j, h_k\}_1 = 0$ and $\{h_j, h_k\}_{(\lambda)} = 0$, for any $\lambda\in \R\cup \{\infty\}$.
\end{proof}

\begin{proposition}\label{hierarchy 2}
If $(g_k)_{k \in \N}$, $(h_k)_{k \in \N}$ are two bi-Hamiltonian hierarchies, then all functions are in bi-involution under the assumption that one of the two sequences starts from a Casimir invariant of $\{\cdot, \cdot\}_0.$
\end{proposition}
\begin{proof}
Suppose that $g_0$ is such a Casimir. Then, for any $j,k\in \N$,
\begin{eqnarray*}
\{g_j, h_k\}_0 & = & \{g_{j-1}, h_k\}_1 = \{g_{j-1}, h_{k+1}\}_0 = \ldots \\
  & = & \{g_0, h_{k+j}\}_0 = 0.
\end{eqnarray*}
Thus,
\begin{equation*}
\{g_j, h_k\}_0 = \{g_{j-1},h_k\}_1 =0, \quad \mathrm{for} \;\,\mathrm{all}\;j,k\in \N.
\end{equation*}
Furthermore, the above yield that $\{g_j, h_k\}_{(\lambda)}=0$, for any $\lambda\in \R\cup \{\infty\}$.
\end{proof}

\vspace{2mm}

We remark that, if $\Pi_0$ is degenerate and a bi-Hamiltonian hierarchy $(C_k)_{k \in \N}$ of $(M, \Pi_0, \Pi_1)$ starts from a Casimir function $C_0$ of $\Pi_0$, then an immediate consequence of relations (\ref{Lenard-Magri-1}) is the next:
\begin{corol}\label{Corollary-coefficient-Casimir}
The formal Laurent series
\begin{equation}\label{laurent-series}
C(\lambda) = C_0 + C_1\lambda^{-1} +\cdots + C_k\lambda^{-k}+\cdots
\end{equation}
is a Casimir function of the Poisson pencil $\Pi_{(\lambda)}$, $\lambda \in \R\cup\{\infty\}$. Reciprocally, if a Casimir $C(\lambda)$ of $\{\cdot,\cdot\}_{(\lambda)}$ can be developed in a formal Laurent series as in (\ref{laurent-series}), then its coefficients consist a bi-Hamiltonian hierarchy.
\end{corol}
\begin{proof} Effectively, we have
\begin{eqnarray*}
\{\cdot, C(\lambda)\}_{(\lambda)} = 0 & \Leftrightarrow & \{\cdot,\, C_0 + C_1\lambda^{-1} +\cdots + C_k\lambda^{-k}+\cdots\}_1 \\
& & - \;\lambda\{\cdot,\, C_0 + C_1\lambda^{-1} +\cdots + C_k\lambda^{-k}+\cdots\}_0 =0 \\
& \Leftrightarrow & -\; \lambda\{\cdot, C_0\cdot\}_0 + (\{\cdot, C_0\}_1-\{\cdot,C_1\}_0) + (\{\cdot, C_1\}_1-\{\cdot,C_2\}_0)\lambda^{-1}\\
& &  + \;\ldots\; + (\{\cdot, C_k\}_1-\{\cdot,C_{k+1}\}_0)\lambda^{-k} +\; \ldots \;= 0.
\end{eqnarray*}
The last equality is true if and only if $\{\cdot, C_0\}_0 = 0$ and $\{\cdot, C_k\}_1-\{\cdot,C_{k+1}\}_0 =0$, for any $k\in \N$.
\end{proof}

\vspace{2mm}

Finally, we note that, in the particular case where inside the family $(C_k)_{k \in \N}$ of coefficients in (\ref{laurent-series}) there is a Casimir function $C_n$ of $\Pi_1$, then the Casimir $C(\lambda)$
of $\Pi_{(\lambda)}$ takes the form of a polynomial
\begin{equation*}
C(\lambda) = C_0\lambda^n + C_1\lambda^{n-1} +\cdots + C_{n-1}\lambda + C_n.
\end{equation*}

\vspace{2mm}

The above mentioned indicate that a process for building bi-Hamiltonian hierarchies on $(M,\Pi_0,\Pi_1)$ and, by consequence, functions in involution with respect to $\Pi_0$ and $\Pi_1$, is to look for the Casimir invariants $C(\lambda)$ of the Poisson pencil $(\Pi_{(\lambda)})_{\lambda \in \R\cup\{\infty\}}$, written in formal series of $\lambda^{-1}$, which are deformations of Casimir functions of $\Pi_0$. The main results of the theory are due to I. M. Gel'fand and I. Zakharevich \cite{GZ-Toda} (sections 10 and 11) and can be formalized in the following theorem.
\begin{theorem}
Let $(M,\Pi_0,\Pi_1)$ be a $2r+k$-dimensional bi-Hamiltonian manifold and $(\Pi_{(\lambda)})_{\lambda \in \R\cup \{\infty\}}$ the Poisson pencil on $M$ produced by $(\Pi_0,\Pi_1)$. We assume that,
for almost all the values of the parameter $\lambda$ and almost everywhere on $M$, $\Pi_{(\lambda)}$ has rank $2r$. Then, on a neighborhood of a generic point of $M$ there exist $k$ functionally
independent Casimir functions of the Poisson pencil $(\Pi_{(\lambda)})_{\lambda \in \R \cup\{\infty\}}$ depended on $\lambda$ which are power series in the parameter $\lambda^{-1}$.
\end{theorem}
\begin{sk-proof}
Consider $k$ functionally independent Casimir functions $C_0^1,\ldots,C_0^k$ of $\Pi_0$ on a neighborhood of a generic point of $M$. Under the assumption of the compatibility of $\Pi_0$ with $\Pi_1$ and some topological assumptions, we can conclude that there exists a locally defined solution $(C_s^i)_{s\in \N}$ to the recursion relations (\ref{Lenard-Magri-1}) with $C_0^i$ as the initial data, for any $i=1,\ldots,k$. Also, there are numbers $t_1,\ldots,t_k$ such that, for each $i=1,\ldots,k$, we can extract a collection $(C_s^i)_{0\leq s \leq t_i}$ of functionally independent functions from the family $(C_s^i)_{s\in \N}$ and all
functions $C_s^i$, $i=1,\ldots,k$, $s\in \N$, depend functionally on this collection. Then the functions $C^i(\lambda) = \sum_{s\in \N}C_s^i\lambda^{-s}$ are Casimir invariants
of $\{\cdot,\cdot\}_{(\lambda)}$, $\lambda \in \R\cup\{\infty\}$, depending on $(C_s^i)_{i=1,\ldots,k, \,0\leq s \leq t_i}$.
\end{sk-proof}

\begin{remark}\label{remark casimirs}
{\rm If, as it frequently happens in the typical situation and applications, under some technical assumptions, each hierarchy $(C_s^i)_{s\in \N}$, $i=1,\ldots,k$, starts from a Casimir invariant $C_0^i$ of $\Pi_0$ and terminates with a Casimir function $C_{r_i}^i$ of $\Pi_1$, then the polynomials $C^i(\lambda) = C_0^i\lambda^{r_i} + C_1^i\lambda^{r_i-1} +\cdots + C_{r_i-1}^i\lambda + C_{r_i}^i$, $i=1,\ldots,k$, whose degrees $r_i$, $i=1,\ldots,k$, satisfy the relation $r_1+\cdots +r_k = r$, are Casimir functions of $\{\cdot,\cdot\}_{(\lambda)}$ and any other Casimir invariant of the Poisson pencil $\{\cdot,\cdot\}_{(\lambda)}$ is functionally generated by $C^1(\lambda), \ldots, C^k(\lambda)$. Notice that polynomial Casimirs of degree $0$ are nothing else than common Casimir functions of the brackets $\{\cdot,\cdot\}_0$ and $\{\cdot,\cdot\}_1$. Hence, for a Poisson pencil $(\Pi_{(\lambda)})_{\lambda \in \R\cup \{\infty\}}$ of corank $k$ almost everywhere on a $2r+k$--dimensional manifold and for almost any value of $\lambda$, it is not inconceivable to suppose that it possesses $k$ functionally independent polynomial Casimirs with the properties described in this remark.}
\end{remark}

From the above presentation we keep that:
\begin{itemize}
\item[-]
The coefficients of a polynomial Casimir of a Poisson pencil are in bi-involution.
\item[-]
All the coefficients of two polynomial Casimirs of a Poisson pencil are also in bi-involution.
\end{itemize}

\section{Poisson brackets with prescribed involutive family of functions}\label{section-central}
Let $\mathfrak{F}=(f_1,\ldots,f_{r+k})$ be a $r+k$--tuple of smooth functions on a $2r+k$--dimensional smooth manifold $M$, functionally independent almost everywhere. We want to construct Poisson brackets $\{\cdot,\cdot\}$ on $\C$ of rank at most $2r$ with respect to which the elements of $\mathfrak{F}$ are pairwise in involution. Our process for the study of this problem is based on the properties of polynomial Casimir functions of a Poisson pencil presented in paragraph \ref{Casimir of pencil} and on the results of \cite{dam-pet} reviewed in \ref{Poisson given Casimirs}. Our basic idea is to:
\begin{enumerate}
\item
Construct $k$ polynomials of a parameter $\lambda\in \R\cup \{\infty\}$ with coefficients from $\mathfrak{F}$. For this, we rename the elements of the collection $\mathfrak{F}$ and formulate, using every element of $\mathfrak{F}$ only once, $k$ polynomials of the form
\begin{equation*}
F^i(\lambda) =  \lambda^{r_i}f^i_0 + \lambda^{r_i-1}f^i_1 + \ldots + \lambda f^i_{r_i-1} + f^i_{r_i},  \quad i=1, \ldots, k,
\end{equation*}
with $f^i_j \in \mathfrak{F}$ and $\sum_{i=1}^k r_i = r$. We note that polynomials $F^i$ of degree $r_i=0$ are "constant" polynomials, i.e., $F^i(\lambda)=f^i_{r_i}=f^i_0$.
\item
Extend our method of constructing Poisson brackets with given Casimirs \cite{dam-pet} to Poisson pencils. This procedure allows us to deal with compatible Poisson structures in a constructive way which is a problem already posed by J. Grabowski and al. in \cite{Grab93}.
\end{enumerate}
Therefore, our purpose is to build a Poisson pencil $(\Pi_{(\lambda)})_{\lambda \in \R\cup\{\infty\}}$ of rank $2r$, for almost all the values of the parameter $\lambda$ and almost everywhere on $M$, admitting the polynomials $F^1(\lambda), \ldots, F^k(\lambda)$ as polynomial Casimir functions. As we have noted in Remark \ref{remark casimirs}, polynomials of degree $0$ will be common Casimir functions of the brackets $\{\cdot,\cdot\}_{(\lambda)}$, $\lambda\in \R\cup\{\infty\}$. As in \cite{dam-pet}, we begin by discussing the problem on even-dimensional manifolds and in the next subsection we continue by extending the results on odd-dimensional manifolds.

\subsection{On even-dimensional manifolds}\label{even}
We suppose that $\dim M=2r+k$ is even, i.e., $k=2l$ is also even, and, for technical reasons, that $M$ is endowed with a nondegenerate bivector field $\Lambda$ such that
\begin{equation}\label{condition-Lambda-0-1}
\left\{
\begin{array}{l}
F_0 = \langle df_0^1\wedge \ldots \wedge df_0^k, \, \displaystyle{\frac{\Lambda^l}{l!}}\rangle = \langle \frac{\omega^l}{l!},\, X_{f_0^1}\wedge \ldots \wedge X_{f_0^k}\rangle \neq 0,  \\
\\
F_r = \langle df_{r_1}^1\wedge \ldots \wedge df_{r_k}^k, \, \displaystyle{\frac{\Lambda^l}{l!}}\rangle = \langle \frac{\omega^l}{l!},\, X_{f_{r_1}^1}\wedge \ldots \wedge X_{f_{r_k}^k}\rangle \neq 0
\end{array}
\right.
\end{equation}
are verified on an open and dense subset $\mathcal{U}$ of $M$.\footnote{We can easily prove that such a structure always existe at least locally.} In (\ref{condition-Lambda-0-1}), $\omega$ is the nondegenerate $2$-form on $M$ defined by $\Lambda$ and $X_{f^i_j}=\Lambda^\#(df_j^i)$, $i=1,\ldots, k$, $j=1, \ldots, r_i$, are the vector fields on $M$ associated to $f^i_j$ via $\Lambda^\#$. Then the relation
\begin{equation*}
F(\lambda) = \langle dF^1(\lambda)\wedge \ldots \wedge dF^k(\lambda), \, \frac{\Lambda^l}{l!}\rangle = \langle \frac{\omega^l}{l!},\, X_{F^1(\lambda)}\wedge \ldots \wedge X_{F^k(\lambda)} \rangle \neq 0,
\end{equation*}
where $X_{F^i(\lambda)}=\Lambda^\#(dF^i(\lambda))$, $i=1,\ldots,k$, also holds on $\mathcal{U}$, for every $\lambda \in \R\cup\{\infty\}$, because
\begin{equation*}
dF^1(\lambda)\wedge \ldots \wedge dF^k(\lambda) = \lambda^rdf_0^1\wedge \ldots \wedge df_0^k + \ldots + \ldots + df_{r_1}^1\wedge \ldots \wedge df_{r_k}^k
\end{equation*}
and $F(\lambda) = \lambda^rF_0 + \ldots + F_r$ with $F_0\neq 0$, $F_r\neq 0$ on $\mathcal{U}$.

\vspace{2mm}

Thereafter, we consider the distributions $D_0 = \langle X_{f_0^1}, \ldots, X_{f_0^k} \rangle$, $D_1= \langle X_{f_{r_1}^1}, \ldots, X_{f_{r_k}^k} \rangle$ and $D_{(\lambda)} = \langle X_{F^1(\lambda)}, \ldots , X_{F^k(\lambda)}\rangle$ on $M$, of rank at most $k$, their annihilators $D_0^\circ$, $D_1^\circ$ and $D_{(\lambda)}^\circ$, and their orthogonal distributions with respect to $\omega$ $\mathrm{orth}_\omega D_0$, $\mathrm{orth}_\omega D_1$ and $\mathrm{orth}_\omega D_{(\lambda)}$.{\footnote{We note that $D_0 \neq D_{(0)}$ and $D_1 \neq D_{(1)}$. Precisely, $D_{(0)} = D_1$.}} Since $\det(\{f_0^i, f_0^j\}_{\omega})=F_0^2 \neq 0$, $\det(\{f_{r_i}^i, f_{r_j}^j\}_{\omega})=F_r^2 \neq 0$ and $\det(\{F^i(\lambda), F^j(\lambda)\}_{\omega})=F(\lambda)^2 \neq 0$ on $\mathcal{U}$, the spaces $D_{0x}=D_0\cap T_xM$, $D_{1x}=D_1\cap T_xM$ and $D_{(\lambda)x}=D_{(\lambda)}\cap T_xM$ are symplectic subspaces of $(T_xM,\omega_x)$ at each point $x\in \mathcal{U}$. Thus,
\begin{equation*}
T_xM = D_{0x} \oplus \mathrm{orth}_{\omega_x} D_{0x} = D_{0x} \oplus \Lambda_x^\#(D_{0x}^\circ),
\end{equation*}
\begin{equation*}
T_xM = D_{1x} \oplus \mathrm{orth}_{\omega_x} D_{1x} = D_{1x} \oplus \Lambda_x^\#(D_{1x}^\circ)
\end{equation*}
and
\begin{equation*}
T_xM = D_{(\lambda)x} \oplus \mathrm{orth}_{\omega_x} D_{(\lambda)x} = D_{(\lambda)x} \oplus \Lambda_x^\#(D_{(\lambda)x}^\circ),
\end{equation*}
where $D_{0x}^\circ = D_0^\circ \cap T_x^\ast M$, $D_{1x}^\circ = D_1^\circ \cap T_x^\ast M$ and $D_{(\lambda)x}^\circ = D_{(\lambda)}^\circ \cap T_x^\ast M$. Also,
\begin{equation*}
T_x^\ast M = D_{0x}^\circ  \oplus (\Lambda_x^\#(D_{0x}^\circ))^\circ = D_{0x}^\circ \oplus \langle df_0^1, \ldots, df_0^k\rangle_x,
\end{equation*}
\begin{equation*}
T_x^\ast M = D_{1x}^\circ  \oplus (\Lambda_x^\#(D_{1x}^\circ))^\circ = D_{1x}^\circ \oplus \langle df_{r_1}^1, \ldots, df_{r_k}^k\rangle_x
\end{equation*}
and
\begin{equation*}
T_x^\ast M = D_{(\lambda)x}^\circ  \oplus (\Lambda_x^\#(D_{(\lambda)x}^\circ))^\circ = D_{(\lambda)x}^\circ \oplus \langle dF^1(\lambda), \ldots, dF^k(\lambda)\rangle_x.
\end{equation*}

\vspace{2mm}

We proceed by choosing on $(M,\Lambda)$ a smooth section $\sigma_0$ of $\bigwedge^2D_0^\circ$ of maximal rank on $\mathcal{U}$ and a smooth section $\sigma_1$ of $\bigwedge^2D_1^\circ$ of maximal rank on $\mathcal{U}$ also, verifying the following conditions
\begin{equation}\label{sigma}
\left\{
\begin{array}{l}
\delta(\sigma_0\wedge \sigma_0) = 2\sigma_0\wedge\delta(\sigma_0),  \\
\\
\delta(\sigma_1\wedge \sigma_1) = 2\sigma_1\wedge\delta(\sigma_1),     \\
\\
\delta(\sigma_0\wedge \sigma_1) = \delta(\sigma_0)\wedge\sigma_1 + \sigma_0\wedge\delta(\sigma_1)
\end{array}
\right.
\end{equation}
and, for any $i=1,\ldots,k$ and $j=1,\ldots,r_i$, the recursion relations
\begin{equation}\label{recursion - sigma - X}
\sigma_0(X_{f^i_{j}}, \cdot) = \sigma_1(X_{f^i_{j-1}}, \cdot).
\end{equation}
As in Proposition \ref{prop-Pi-sigma}, $\delta$ is the codifferential operator on $\Omega(M)$ defined by $\omega$ and the isomorphism $\Lambda^\# : \Omega^p(M) \to \mathcal{V}^p(M)$, $p\in \N$. We then define on $M$ the $1$-parameter family of $2$-forms
\begin{equation*}
\sigma_{(\lambda)} = \sigma_1-\lambda \sigma_0, \quad \quad \lambda \in \R\cup\{\infty\},
\end{equation*}
and prove:

\begin{theorem}\label{central-theorem}
Under the above assumptions, the following statements hold.
\begin{enumerate}
\item
The bivector field $\Pi_0 = \Lambda^\#(\sigma_0)$ on $(M,\Lambda)$ is a Poisson tensor of rank at most $2r$ on $M$ admitting as unique Casimir invariants the functions $f_0^1,\ldots,f_0^k$.\footnote{In the paper, the statement \emph{"the functions are the unique Casimir invariants of a Poisson structure $\Pi$"} means that the algebra of Casimir functions of $\Pi$ is generated by these functions.}
\item
The bivector field $\Pi_1 = \Lambda^\#(\sigma_1)$ on $(M,\Lambda)$ is a Poisson tensor of rank at most $2r$ on $M$ admitting as unique Casimir invariants the functions $f_{r_1}^1,\ldots,f_{r_k}^k$.
\item
The $2$-forms $\sigma_{(\lambda)} = \sigma_1-\lambda \sigma_0$, $\lambda \in \R\cup\{\infty\}$, are smooth sections of $\bigwedge^2D_{(\lambda)}^\circ$ of maximal rank on $\mathcal{U}$ and the bivector fields $\Pi_{(\lambda)}= \Lambda^\#(\sigma_{(\lambda)})$ on $(M,\Lambda)$ define a Poisson pencil of rank at most $2r$ on $M$ admitting as unique polynomial Casimir invariants the polynomials $F^1(\lambda), \ldots, F^k(\lambda)$. The Poisson bracket $\{\cdot, \cdot\}_{(\lambda)}$, $\lambda \in \R\cup \{\infty\}$, on $\C$ associated to $\Pi_{(\lambda)}$ is given by the formula
\begin{equation}\label{bracket-lambda-even}
\{\cdot , \cdot\}_{(\lambda)} \Omega = -\frac{1}{F(\lambda)}d\cdot \wedge\, d\cdot \wedge\big(\sigma_{(\lambda)} + \frac{g_{(\lambda)}}{r-1}\omega\big) \wedge \frac{\omega^{r-2}}{(r-2)!}\wedge dF^1(\lambda)\wedge\ldots \wedge dF^k(\lambda),
\end{equation}
where $\Omega$ is the volume form on $M$ defined by $\omega$ and $g_{(\lambda)} = i_{\Lambda}\sigma_{(\lambda)}$.
\item
The functions of the family $\mathfrak{F}=(f^1_0,\ldots,f^1_{r_1},\ldots, f^k_0,\ldots, f^k_{r_k})$ are in involution with respect to any Poisson bracket $\{\cdot, \cdot\}_{(\lambda)}$, $\lambda \in \R\cup\{\infty\}$.
\end{enumerate}
\end{theorem}
\begin{proof}
According to Proposition \ref{prop-Pi-sigma}, $\Pi_0 = \Lambda^\#(\sigma_0)$ is Poisson because $\sigma_0$ verifies the first condition of (\ref{sigma}). By construction, $f_0^1,\ldots,f_0^k$ are the unique Casimir invariants of $\Pi_0$. Effectively, for any $i=1,\ldots,k$,
\begin{equation*}
\Pi_0^\#(df_0^i, \cdot) = \Lambda^\#(\sigma_0)(df_0^i, \cdot) = \sigma_0(\Lambda^\#(df_0^i), \Lambda^\#(\cdot)) = \sigma_0(X_{f_0^i}, \Lambda^\#(\cdot))=0,
\end{equation*}
since $\sigma_0$ and $X_{f_0^i}$ are sections of $\bigwedge^2D_0^\circ$ and $D_0$, respectively. If $f$ is any other Casimir of $\Pi_0$, then $X_f = \Lambda^\#(df)$ must be a section of $D_0$, fact which implies that $f$ depends functionally on $f_0^1,\ldots,f_0^k$. Thus, $f_0^1,\ldots,f_0^k$ are the generators of the algebra of Casimirs of $\Pi_0$ and the first assertion of the Theorem is proved. With a completely analogous argumentation,
we establish also the second assertion.

In order to show the third statement, we remark that, for any $i=1,\ldots, k$,
\begin{eqnarray*}
X_{F^i(\lambda)} & = & \Lambda^\#(dF^i(\lambda)) = \lambda^{r_i}\Lambda^\#(df_0^i) + \lambda^{r_i-1}\Lambda^\#(df_1^i) + \ldots + \lambda\Lambda^\#(df_{r_i-1}^i) + \Lambda^\#(df_{r_i}^i) \\
& = & \lambda^{r_i}X_{f_0^i} + \lambda^{r_i-1}X_{f_1^i} + \ldots + \lambda X_{f_{r_i-1}^i} + X_{f_{r_i}^i}
\end{eqnarray*}
and
\begin{eqnarray}\label{sigma-X-lambda}
\sigma_{(\lambda)}(X_{F^i(\lambda)}, \cdot ) & = & \lambda^{r_i}\sigma_1(X_{f_0^i},\cdot) + \lambda^{r_i-1}\sigma_1(X_{f_1^i}, \cdot) + \ldots + \lambda \sigma_1(X_{f_{r_i-1}^i},\cdot) + \sigma_1(X_{f_{r_i}^i},\cdot) \nonumber \\
& & - \lambda^{r_i+1}\sigma_0(X_{f_0^i},\cdot) - \lambda^{r_i}\sigma_0(X_{f_1^i}, \cdot) - \ldots - \lambda^2 \sigma_0(X_{f_{r_i-1}^i},\cdot) -\lambda \sigma_0(X_{f_{r_i}^i},\cdot) \nonumber \\
& = & - \lambda^{r_i+1}\sigma_0(X_{f_0^i},\cdot) + \lambda^{r_i}(\sigma_1(X_{f_0^i},\cdot) - \sigma_0(X_{f_1^i}, \cdot)) + \ldots \nonumber \\
& & + \lambda (\sigma_1(X_{f_{r_i-1}^i},\cdot)-\sigma_0(X_{f_{r_i}^i},\cdot) ) + \sigma_1(X_{f_{r_i}^i},\cdot) \nonumber \\
& = & 0.
\end{eqnarray}
The last equality is true because the recursion relations (\ref{recursion - sigma - X}) hold and, $\sigma_0$ and $\sigma_1$ are assumed to be sections of $\bigwedge^2D_0^\circ$ and $\bigwedge^2D_1^\circ$, respectively.
The above result indicates that $\sigma_{(\lambda)}$ is a section of $\bigwedge^2D_{(\lambda)}^\circ$, while conditions (\ref{sigma}), which are true by our assumptions, assure us that, for any $\lambda\in \R\cup\{\infty\}$, $\Pi_{(\lambda)}= \Lambda^\#(\sigma_{(\lambda)}) =\Lambda^\#(\sigma_1 - \lambda \sigma_0)=\Lambda^\#(\sigma_1) - \lambda \Lambda^\#(\sigma_0)=\Pi_1 -\lambda \Pi_0$ is a Poisson tensor on $M$. In fact,
\begin{eqnarray*}
\Pi_{(\lambda)} \; \mathrm{is}\;\mathrm{Poisson} & \Leftrightarrow & \delta(\sigma_{(\lambda)}\wedge \sigma_{(\lambda)}) = 2\sigma_{(\lambda)}\wedge \delta(\sigma_{(\lambda)}) \\
& \Leftrightarrow & \delta(\sigma_1\wedge \sigma_1) - 2\lambda \delta(\sigma_0\wedge\sigma_1) + \lambda^2\delta(\sigma_0\wedge\sigma_0) = \\
& & 2\sigma_1\wedge \delta(\sigma_1)-2\lambda(\sigma_0\wedge\delta(\sigma_1) + \delta(\sigma_0)\wedge\sigma_1) + 2\lambda^2\sigma_0\wedge\delta(\sigma_0) \\
& \Leftrightarrow & (\ref{sigma}) \; \mathrm{hold}.
\end{eqnarray*}
The compatibility condition $[\Pi_0,\Pi_1]=0$ of $\Pi_0$ with $\Pi_1$ is equivalent to the third relation of (\ref{sigma}). Moreover, equation (\ref{sigma-X-lambda}) implies that $F^1(\lambda), \ldots, F^k(\lambda)$ are polynomial Casimirs functions of the Poisson pencil $\Pi_{(\lambda)} = \Pi_1 - \lambda \Pi_0$, $\lambda\in \R\cup\{\infty\}$, since
\begin{equation*}
\Pi_{(\lambda)}(dF^i(\lambda), \cdot) = \Lambda^\#(\sigma_{(\lambda)})(dF^i(\lambda), \cdot) = \sigma_{(\lambda)}(\Lambda^\#(dF^i(\lambda)), \Lambda^\#(\cdot)) = \sigma_{(\lambda)}(X_{F^i(\lambda)}, \cdot )\stackrel{(\ref{sigma-X-lambda})}{=}0.
\end{equation*}
As for the first assertion, we prove that any other polynomial Casimir function of $\Pi_{(\lambda)}$ depends functionally on $F^1(\lambda), \ldots, F^k(\lambda)$, thus $\Pi_{(\lambda)}$ is of rank at most $2r$ on $M$. Applying the results of \cite{dam-pet}, presented in paragraph \ref{Poisson given Casimirs}, to the family of functions $(F^1(\lambda),\ldots, F^k(\lambda))$, we conclude that the brackets $\{\cdot,\cdot\}_{(\lambda)}$ can be written as in formula (\ref{bracket-lambda-even}).

Finally, taking into account Corollary \ref{Corollary-coefficient-Casimir} and the fact that $f^i_0$, $i=1,\ldots,k$, are Casimirs of $\{\cdot,\cdot\}_0$, we obtain that the coefficients $f^i_0,\ldots,f^i_{r_i}$ of each polynomial Casimir $F^i(\lambda)$ of $\Pi_{(\lambda)}$, $i=1,\ldots,k$, constitute a bi-Hamiltonian hierarchy. By Proposition \ref{hierarchy 2}, we then get that the functions of the family $\mathfrak{F}=(f^1_0,\ldots,f^1_{r_1},\ldots, f^k_0,\ldots, f^k_{r_k})$ are in bi-involution with respect to $\{\cdot,\cdot\}_0$ and $\{\cdot,\cdot\}_1$. Thus, the fourth statement of Theorem is proved.
\end{proof}

\subsection{On odd-dimensional manifolds}\label{odd}
In the case where $\dim M = 2r+k$ is odd, i.e., $k=2l+1$, we remark that any Poisson bracket $\{\cdot,\cdot\}$ on $\C$ having the elements of $\mathfrak{F}$ in involution can be considered as a Poisson bracket on $C^\infty(M')$, $M'=M\times \R$, with respect to which the functions of $\mathfrak{F}'=\mathfrak{F}\cup \{s\}$ pairwise commute. The function $s$ is the canonical coordinate on the factor $\R$ and, in this setting, is also a Casimir invariant of $\{\cdot,\cdot\}$. The converse statement is also true. Thus, the problem of constructing Poisson brackets $\{\cdot,\cdot\}$ on $\C$ for which $\mathfrak{F}$ is involutive, is equivalent to that of constructing Poisson brackets $\{\cdot,\cdot\}$ on $C^\infty(M')$ for which $\mathfrak{F}'$ is involutive and, in particular, $s$ is a Casimir; a framework that was completely studied in the previous subsection. In what follows and using the results of Section \ref{even}, our aim is to establish Poisson pencils on $M$ having the elements of $\mathfrak{F}$ in involution and a formula analogous to (\ref{bracket-lambda-even})
for their brackets. For this, we need to adapt the technique of the proof of Theorem \ref{central-theorem} to the situation described above. The latter can be realized by choosing, for the role of $(\omega, \Lambda)$ on $M'$, an (almost) symplectic form $\omega'$ and its associated bivector field $\Lambda'$ defined by an (almost) cosymplectic structure $(\vartheta, \Theta)$ on $M$. Before we proceed, let us recall the notion of \emph{almost cosymplectic} structure on an odd-dimensional manifold and some of their properties \cite{lib1}, \cite{lch2}.

\vspace{2mm}

An \emph{almost cosymplectic} structure on a smooth odd-dimensional manifold $M$, $\dim M = 2n+1$, is defined by a pair $(\vartheta,\Theta)\in \Omega^1(M)\times \Omega^2(M)$ such that $\Omega=\vartheta \wedge \Theta^n$ is a volume form on $M$, fact that yields that $\Theta$ is of constant rank $2n$ on $M$. Thus, $\ker \vartheta$ and $\ker \Theta$ are complementary subbundles of $TM$ called, respectively, the \emph{horizontal bundle} and the \emph{vertical bundle}, and their annihilators are complementery subbundles of $T^\ast M$. Moreover, it is well known \cite{lch2} that $(\vartheta,\Theta)$ gives rise to a transitive \emph{almost Jacobi} structure $(\Lambda,E) \in \mathcal{V}^2(M)\times \mathcal{V}^1(M)$ on $M$ such that
\begin{equation*}
i_{E}\vartheta = 1 \quad  \mathrm{and}  \quad i_{E}\Theta = 0,
\end{equation*}
\begin{equation*}
\Lambda^\#(\vartheta) = 0 \quad \mathrm{and} \quad i_{\Lambda^\#(\zeta)}\Theta = -(\zeta - \langle \zeta, E\rangle \vartheta), \quad \mathrm{for}\;\;\mathrm{all}\;\; \zeta \in \Omega^1(M),
\end{equation*}
and conversely. We have, $\ker \vartheta = \mathrm{Im}\Lambda^\#$ and $\ker \Theta = \langle E\rangle$. So, $TM =\mathrm{Im}\Lambda^\# \oplus \langle E\rangle$ and $T^\ast M = \langle E\rangle^\circ \oplus \langle \vartheta\rangle$. The sections of $\langle E\rangle^\circ$ are called \emph{semi-basic} forms and $\Lambda^\#$ is an isomorphism from the $\C$-module of semi-basic $1$-forms to  the $\C$-module of horizontal vector fields. This isomorphism can be  extended, as in (\ref{def-extension P}), to an isomorphism, also denoted by $\Lambda^\#$, from the $\C$-module of semi-basic $p$-forms on the $\C$-module of horizontal $p$-vector fields. Furthermore, we note that $(\vartheta,\Theta)$ defines on $M'=M\times \R$ an almost symplectic structure $\omega' = \Theta + ds \wedge \vartheta$ whose corresponding nondegenerate bivector field is $\Lambda' = \Lambda + \displaystyle{\frac{\partial}{\partial s}}\wedge E$ and it determines, as in (\ref{def-extension P}), an isomorphism $\Lambda'\,^\#$ from the space $\Omega^p(M')$ of smooth $p$-forms on $M'$ to the space $\mathcal{V}^p(M')$ of smooth $p$-vector fields on $M'$. Thus, any bivector field $\Pi'$ on $M'$ can be viewed as the image by $\Lambda'\,^\#$ of a $2$-form $\sigma'$ on $M'$.
In particular, any bivector field $\Pi$ on $M$, which is viewed as a bivector field on $M'$ independent of $s$ and without a term of type $X\wedge \frac{\partial}{\partial s}$, i.e., $ds$ belongs to the kernel of $\Pi^\#$, is written on $M'$ as $\Pi =\Lambda'\,^\#(\sigma')$ with $\sigma'$ a $2$-form on $M'$ of type
\begin{equation}\label{type sigma'}
\sigma' = \sigma + \tau \wedge ds,
\end{equation}
where $\sigma$ and $\tau$ are, respectively, a semi-basic $2$-form and a semi-basic $1$-form on $M$.

\vspace{2mm}

Now, we consider on $(M',\mathfrak{F}')$ the polynomials
\begin{equation*}
F^i(\lambda) =  \lambda^{r_i}f^i_0 + \lambda^{r_i-1}f^i_1 + \ldots + \lambda f^i_{r_i-1} + f^i_{r_i},  \quad i=1, \ldots, k,
\end{equation*}
with $\sum_{i=1}^k r_i = r$, and \begin{equation*}
F^{k+1}(\lambda)=s, \end{equation*}
where the function $s$ has, simultaneously, the role of $f^{k+1}_0$ and that of $f^{k+1}_{r_{k+1}}$ with $r_{k+1}=0$. On the other hand, we consider on $M$ an almost Jacobi structure $(\Lambda,E)$ such that
\begin{equation*}
F_0=\langle df^1_0\wedge \ldots \wedge df^k_0, \, E\wedge \frac{\Lambda^l}{l!}\rangle \neq 0 \quad \mathrm{and} \quad F_r=\langle df^1_{r_1}\wedge \ldots \wedge df^k_{r_k}, \, E\wedge \frac{\Lambda^l}{l!}\rangle \neq 0
\end{equation*}
on an open and dense subset $\mathcal{U}$ of $M$.{\footnote{We note that, as in the case of even-dimensional manifolds, such a choise of $(\Lambda, E)$ on $M$ is always possible at least locally.}} Hence, the relation
\begin{eqnarray*}
F(\lambda) & = & \langle dF^1(\lambda)\wedge \ldots \wedge dF^k(\lambda), \, E\wedge \frac{\Lambda^l}{l!} \rangle    \nonumber \\
& = & \langle \lambda^r df^1_0\wedge \ldots \wedge df^k_0 + \ldots +  df^1_{r_1}\wedge \ldots \wedge df^k_{r_k}, \, E\wedge \frac{\Lambda^l}{l!} \rangle \nonumber \\
& = & \lambda^r F_0 + \ldots + F_r \neq 0
\end{eqnarray*}
also holds on $\mathcal{U}$. Let $(\vartheta, \Theta)$ be the almost cosymplectic structure on $M$ associated to $(\Lambda,E)$ and $\Lambda'=\Lambda + \frac{\partial}{\partial s}\wedge E$
the corresponding almost Poisson tensor on $M'$ with almost symplectic form $\omega'=\Theta + ds\wedge\vartheta$. Since, for any $m=1,\ldots,r+l+1$,
\begin{equation*}\label{volume'}
\frac{\omega'\,^m}{m!} = \frac{\Theta^m}{m!} + ds\wedge \vartheta \wedge \frac{\Theta^{m-1}}{(m-1)!} \quad \mathrm{and} \quad \frac{\Lambda'\,^m}{m!} = \frac{\Lambda^m}{m!} +
\frac{\partial}{\partial s}\wedge E \wedge \frac{\Lambda^{m-1}}{(m-1)!},
\end{equation*}
it is clear that
\begin{eqnarray*}\label{F_0-odd-even}
\lefteqn{\langle df^1_0\wedge \ldots \wedge df_0^k\wedge ds,\; \frac{\Lambda'\,^{l+1}}{(l+1)!} \rangle }\nonumber \\
& = & \langle df_0^1\wedge \ldots \wedge df_0^k\wedge ds,\; \frac{\Lambda^{l+1}}{(l+1)!} + \frac{\partial}{\partial s}\wedge E \wedge \frac{\Lambda^{l}}{l!}\rangle \nonumber \\
& = & \langle df_0^1\wedge \ldots \wedge df_{0}^k\wedge ds,\; \frac{\partial}{\partial s}\wedge E \wedge \frac{\Lambda^{l}}{l!}\rangle = - F_0 \neq 0
\end{eqnarray*}
and
\begin{eqnarray*}\label{F_r-odd-even}
\lefteqn{\langle df^1_{r_1}\wedge \ldots \wedge df_{r_k}^k\wedge ds,\; \frac{\Lambda'\,^{l+1}}{(l+1)!} \rangle }\nonumber \\
& = & \langle df_{r_1}^1\wedge \ldots \wedge df_{r_k}^k\wedge ds,\; \frac{\Lambda^{l+1}}{(l+1)!} + \frac{\partial}{\partial s}\wedge E \wedge \frac{\Lambda^{l}}{l!}\rangle \nonumber \\
& = & \langle df_{r_1}^1\wedge \ldots \wedge df_{r_k}^k\wedge ds,\; \frac{\partial}{\partial s}\wedge E \wedge \frac{\Lambda^{l}}{l!}\rangle = - F_r \neq 0
\end{eqnarray*}
on the open and dense subset $\mathcal{U}'=\mathcal{U} \times \R$ of $M'=M\times \R$. Consequently,
\begin{eqnarray*}
\lefteqn{\langle dF^1(\lambda)\wedge \ldots \wedge dF^k(\lambda)\wedge dF^{k+1}(\lambda), \, \frac{\Lambda'\,^{l+1}}{(l+1)!}\rangle } \\
& = & \langle \lambda^r df^1_0\wedge \ldots \wedge df_0^k\wedge ds +\,\ldots\,+ df^1_{r_1}\wedge \ldots \wedge df_{r_k}^k\wedge ds, \; \frac{\Lambda'\,^{l+1}}{(l+1)!}\rangle \\
& = & -\lambda^rF_0 - \ldots - F_r = -F(\lambda) \neq 0
\end{eqnarray*}
on $\mathcal{U}'$, also. Therefore, the distributions $D_0'=\langle X'_{f_0^1}, \ldots, X'_{f_0^{k+1}}\rangle$, $D_1'= \langle X'_{f_{r_1}^1}, \ldots, X'_{f_{r_{k+1}}^{k+1}}\rangle$ and $D_{(\lambda)}' = \langle X'_{F^1(\lambda)}, \ldots, X'_{F^{k+1}(\lambda)}\rangle$ of rank $k+1 = 2l+2$ on $M'$ generated, respectively, by the families of vector fields
\begin{equation*}
X'_{f^i_0} = \Lambda'\,^\#(df^i_0) = \Lambda^\#(df^i_0) - \langle df^i_0, E\rangle\frac{\partial}{\partial s}, \; i=1,\ldots,k, \;\; \mathrm{and}  \;\; X'_{f^{k+1}_0}=\Lambda'\,^\#(ds)=E,
\end{equation*}
\begin{equation*}
X'_{f^i_{r_i}} = \Lambda'\,^\#(df^i_{r_i}) = \Lambda^\#(df^i_{r_i}) - \langle df^i_{r_i}, E \rangle\frac{\partial}{\partial s}, \; i=1,\ldots,k, \;\; \mathrm{and}  \;\;  X'_{f^{k+1}_{r_{k+1}}}=\Lambda'\,^\#(ds)=E,
\end{equation*}
and
\begin{equation*}
X'_{F^i(\lambda)} = \Lambda'\,^\#(dF^i(\lambda)) = \Lambda^\#(dF^i(\lambda)) - \langle dF^i(\lambda), E \rangle\frac{\partial}{\partial s}, \; i=1,\ldots,k, \;\; \mathrm{and} \;\; X'_{F^{k+1}(\lambda)}=E,
\end{equation*}
have on $\mathcal{U}'$ all the "good" properties of the distributions $D_0$, $D_1$ and $D_{(\lambda)}$ in the setting of Theorem \ref{central-theorem}.

Let $D_0'^\circ$, $D_1'^\circ$ and $D_{(\lambda)}'^\circ$ be the annihilators of $D_0'$, $D_1'$ and $D_{(\lambda)}'$, respectively. We choose a section $\sigma_0'$ of $\bigwedge^2 D_0'^\circ$
and a section $\sigma_1'$ of $\bigwedge^2 D_1'^\circ$, both of maximal rank $2r$ on $\mathcal{U}'$ and of type (\ref{type sigma'}), i.e.,
\begin{equation*}
\sigma_0' = \sigma_0 + \tau_0 \wedge ds \quad \mathrm{and}    \quad  \sigma_1' = \sigma_1 + \tau_1 \wedge ds,
\end{equation*}
with $\sigma_0$, $\sigma_1$ are semi-basic $2$-forms on $M$ and $\tau_0$, $\tau_1$ are semi-basic $1$-forms on $M$, satisfying the conditions
\begin{equation}\label{sigma'}
\left\{
\begin{array}{l}
\delta'(\sigma'_0\wedge \sigma'_0) = 2\sigma'_0\wedge\delta'(\sigma'_0),  \\
\\
\delta'(\sigma'_1\wedge \sigma'_1) = 2\sigma'_1\wedge\delta'(\sigma'_1),     \\
\\
\delta'(\sigma'_0\wedge \sigma'_1) = \delta'(\sigma'_0)\wedge\sigma'_1 + \sigma'_0\wedge\delta'(\sigma'_1)
\end{array}
\right.
\end{equation}
and the recursion relations
\begin{equation}\label{recursion - sigma' - X}
\sigma'_0(X'_{f^i_{j}}, \cdot) = \sigma'_1(X'_{f^i_{j-1}}, \cdot), \quad \; \mathrm{for}\;\mathrm{any}\; i=1,\ldots,k \;\;\mathrm{and}\;\; j=1,\ldots,r_i.
\end{equation}
The operator $\delta'$ is the codifferential operator on $\Omega(M')$ defined by the isomorphism $\Lambda'\,^\#$ (see, Proposition \ref{prop-Pi-sigma}). The fact that $\sigma_0'$ is a section of $\bigwedge^2 D_0'^\circ$ implies that the semi-basic forms $\sigma_0$ and $\tau_0$ on $M$ satisfy the equations
\begin{equation*}
\sigma_0(X_{f^i_0},\cdot) + \langle df^i_0, E\rangle \tau_0 = 0 \;\;\mathrm{and}\;\; \langle \tau_0, X_{f^i_0}\rangle =0, \;\; i=1,\ldots,k,
\end{equation*}
with $X_{f^i_0} = \Lambda^\#(df^i_0)$. Similarly, the fact that $\sigma_1'$ is a section of $\bigwedge^2 D_1'^\circ$ is equivalent to the system
\begin{equation*}
\sigma_1(X_{f^i_{r_i}},\cdot) + \langle df^i_{r_i}, E\rangle \tau_1 = 0 \;\;\mathrm{and}\;\; \langle \tau_1, X_{f^i_{r_i}}\rangle =0, \;\; i=1,\ldots,k,
\end{equation*}
where $X_{f^i_{r_i}} = \Lambda^\#(df^i_{r_i})$, while the recursion relations (\ref{recursion - sigma' - X}) are equivalent to
\begin{equation*}
\sigma_0(X_{f^i_j},\cdot) + \langle df^i_j, E\rangle \tau_0 = \sigma_1(X_{f^i_{j-1}},\cdot) + \langle df^i_{j-1}, E\rangle \tau_1 \;\; \mathrm{and} \;\; \langle \tau_0 ,X_{f^i_j}\rangle = \langle \tau_1 ,X_{f^i_{j-1}}\rangle,
\end{equation*}
for all $i=1,\ldots, k$ and $j=1,\ldots r_i$. Then, we consider on $M'$ the $1$-parameter family of $2$-forms
\begin{equation*}
\sigma'_{(\lambda)} = \sigma_1'-\lambda\sigma_0' = (\sigma_1-\lambda\sigma_0) + (\tau_1 - \lambda\tau_0)\wedge ds, \quad \lambda \in \R\cup\{\infty\},
\end{equation*}
that determines on $M$ the $1$-parameter families of semi-basic $2$-forms $\sigma_{(\lambda)}=\sigma_1-\lambda\sigma_0$ and $1$-forms $\tau_{(\lambda)}=\tau_1-\lambda\tau_0$, $\lambda \in \R\cup\{\infty\}$.

\vspace{2mm}

Having made all the necessary adaptations of our algorithm and having established the relationships between situations on $(M',\omega',\Lambda')$ and the corresponding ones on $(M, (\vartheta, \Theta),$ $(\Lambda,E))$, we may conclude the following theorem.

\begin{theorem}\label{theorem-odd}
Under the assumptions and notations fixed above, the following statements hold.
\begin{enumerate}
\item
The bivector field $\Pi_0 = \Lambda'\,^\#(\sigma_0') = \Lambda^\#(\sigma_0) + \Lambda^\#(\tau_0)\wedge E$ on $(M, \Lambda, E)$ is Poisson of rank at most $2r$ on $M$ admitting as unique Casimir invariants the functions $f_0^1,\ldots,f_0^k$.
\item
The bivector field $\Pi_1 = \Lambda'\,^\#(\sigma_1') = \Lambda^\#(\sigma_1) + \Lambda^\#(\tau_1)\wedge E$ on $(M, \Lambda, E)$ is Poisson of rank at most $2r$ on $M$ admitting as unique Casimir invariants the functions $f_{r_1}^1,\ldots,f_{r_k}^k$.
\item
The $2$-forms $\sigma'_{(\lambda)} = \sigma'_1-\lambda \sigma'_0$, $\lambda \in \R\cup\{\infty\}$, are smooth sections of $\bigwedge^2 D_{(\lambda)}'^\circ$  of maximal rank on $\mathcal{U}'$ and the bivector fields
\begin{equation*}
\Pi_{(\lambda)}= \Lambda'\,^\#(\sigma'_{(\lambda)}) = \Lambda'\,^\#(\sigma_1') - \lambda \Lambda'\,^\#(\sigma_0') = \Pi_1- \lambda \Pi_0, \;\; \lambda \in \R\cup\{\infty\},
\end{equation*}
on $(M,\Lambda, E)$ define a Poisson pencil of rank at most $2r$ on $M$ admitting as unique polynomial Casimir invariants the polynomials $F^1(\lambda), \ldots, F^k(\lambda)$.
The Poisson bracket $\{\cdot, \cdot\}_{(\lambda)}$ on $\C$ associated to $\Pi_{(\lambda)}$, $\lambda \in \R\cup \{\infty\}$, is given by the formula
\begin{equation}\label{bracket-lambda-odd}
\{\cdot , \cdot\}_{(\lambda)} \Omega = -\frac{1}{F(\lambda)}d\cdot \wedge \, d\cdot \wedge\big(\sigma_{(\lambda)} + \frac{g_{(\lambda)}}{r-1}\Theta\big) \wedge \frac{\Theta^{r-2}}{(r-2)!}\wedge dF^1(\lambda)\wedge\ldots \wedge dF^k(\lambda),
\end{equation}
where $\Omega$ is the volume form on $M$ defined by $(\vartheta, \Theta)$ and $g_{(\lambda)} = i_{\Lambda}\sigma_{(\lambda)}$.
\item
The functions of the family $\mathfrak{F}=(f^1_0,\ldots,f^1_{r_1},\ldots, f^k_0,\ldots, f^k_{r_k})$ are in involution with respect to any Poisson bracket $\{\cdot, \cdot\}_{(\lambda)}$, $\lambda \in \R\cup \{\infty\}$.
\end{enumerate}
\end{theorem}

\section{Examples}\label{examples}
\subsection{From the constants of motion of the Lagrange's top}
The Lagrange's top is an axially symmetric rigid body in a three dimensional space, subject to a constant gravitational field such that the base point of the body-symmetry axis is fixed in the space.
We will give the equations of motion that describe the spinning top in terms of the orthogonal moving coordinate frame, attached to the top, whose axes are the principal inertia axes of the top and its origin coincide with the fixed point. Let $I_1$, $I_2$, $I_3$ be the principal moments of inertia about the fixed point, with $I_1I_2I_3\neq 0$ (because the body is not planar), and $m=(m_1,m_2,m_3)$ the center of gravity of the top. Since the top is symmetric, at least two of the moments of inertia are equal, say $I_1=I_2\neq I_3$, and the center of gravity $m$ lies on the symmetry axis of the body that is characterized by the facts that it passes through the fixed point and corresponds to the equal moments of inertia, thus $m=(0,0,m_3)$ with $m_3\neq 0$, i.e., we exclude the case where $m$ coincides with the fixed point (Euler-Poinsot top). Without loss of generality, by rescaling the variables, we may assume that $I_1=I_2=1$. Then the equations of motion of the top are written as
\begin{equation}\label{Lagrange - top}
\begin{array}{rcl}
\dot{x}_1 & = & l_3 y_3x_2-y_2x_3,  \\
\dot{x}_2 & = & y_1x_3-y_3x_1,  \\
\dot{x}_3 & = & y_2x_1-l_3 y_1x_2,
\end{array}
\quad \quad
\begin{array}{rcl}
\dot{y}_1 & = & m_3x_2+(l_3-1)y_2y_3,   \\
\dot{y}_2 & = & -m_3x_1-(l_3-1)y_1y_3,  \\
\dot{y}_3 & = & 0,
\end{array}
\end{equation}
where $x=(x_1,x_2,x_3)$ is the gravitational field, $y=(y_1,y_2,y_3)$ is the angular momentum, both expressed with respect to the moving coordinate system attached to the top, and $l_3 = 1/I_3$, (for more details, see \cite{adler-vanMoer-Pol}). The above system can be viewed as a vector field on $\R^6$ with coordinates $(x_1,x_2,x_3,y_1,y_2,y_3)$. In these coordinates, the momentum map of the Lagrange's top is $\mathfrak{F}=(f_1,f_2,f_3,f_4)$, where the constants of motion $f_i$ are given by
\begin{eqnarray*}
f_1 & = & x_1^2+x_2^2+x_3^2, \nonumber \\
f_2 & = & x_1y_1+x_2y_2+x_3y_3, \nonumber \\
f_3 & = & \frac{1}{2}(y_1^2 + y_2^2 + l_3y_3^2) + m_3x_3, \nonumber \\
f_4 & = & y_3.
\end{eqnarray*}
The function $f_1$ is the square of the length of the gravity vector, $f_2$ is the angular momentum component in the direction of gravity, $f_3$ is the energy of the top and $f_4$ is the component of the moment of angular momentum in the direction of the symmetry axis. Our aim is, by applying our algorithm, to build Poisson brackets with respect of which the family $\mathfrak{F}$ is involutive.

\vspace{2mm}
We consider the polynomial functions
\begin{equation*}
F^1(\lambda) = \lambda f_1 + f_3 \quad \mathrm{and} \quad F^2(\lambda) = \lambda f_2 +f_4, \quad \lambda \in \R\cup \{\infty\},
\end{equation*}
and we endow $\R^6$ with the nondegenerate bivector field $\Lambda = \sum_{i=1}^3 \displaystyle{\frac{\partial}{\partial x_i}\wedge \frac{\partial}{\partial y_i}}$ which verifies condition (\ref{condition-Lambda-0-1}) on the locus $\{(x,y) \in \R^6 \, / \, x=(x_1,x_2,x_3)\neq 0\}$. In fact, we have
\begin{equation*}
\langle df_1\wedge df_2, \Lambda \rangle = 2x_1^2+2x_2^2+2x_3^2 \neq 0 \quad \;\; \mathrm{and} \quad \;\; \langle df_3\wedge df_4, \Lambda \rangle = m_3\neq 0.
\end{equation*}
The above choice of polynomials $F^1$ and $F^2$ leads us to the search of a Poisson pencil $\Pi_{(\lambda)} = \Pi_1 - \lambda \Pi_0$, $\lambda \in \R\cup\{\infty\}$, defined by a pair $(\Pi_0,\Pi_1)$ of compatible Poisson structures on $\R^6$, whose any element is of rank $4$ almost everywhere on $\R^6$, and such that $f_1, f_2$ are the unique Casimir invariants of $\Pi_0$ and $f_3, f_4$ are the ones of $\Pi_1$. The Hamiltonian vector fields of $f_i$, $i=1,\ldots,4$, with respect to $\Lambda$ are
\begin{eqnarray*}
X_{f_1} & = & 2x_1\frac{\partial}{\partial y_1} + 2x_2\frac{\partial}{\partial y_2} + 2x_3 \frac{\partial}{\partial y_3}, \\
X_{f_2} & = & -x_1\frac{\partial}{\partial x_1} - x_2\frac{\partial}{\partial x_2} - x_3\frac{\partial}{\partial x_3} + y_1\frac{\partial}{\partial y_1} + y_2\frac{\partial}{\partial y_2}+y_3\frac{\partial}{\partial y_3}, \\
X_{f_3} & = & -y_1\frac{\partial}{\partial x_1} - y_2\frac{\partial}{\partial x_2} - l_3y_3\frac{\partial}{\partial x_3} + m_3\frac{\partial}{\partial y_3}, \\
X_{f_4} & = & -\frac{\partial}{\partial x_3}.
\end{eqnarray*}
Following our previous notation, we set $D_0 = \langle X_{f_1}, X_{f_2}\rangle$ and $D_1 = \langle X_{f_3}, X_{f_4}\rangle$. Then, the corresponding annihilators are
\begin{eqnarray*}
\lefteqn{D_0^\circ = \{\alpha_1dx_1 + \alpha_2dx_2 + \alpha_3dx_3 + \beta_1dy_1 + \beta_2dy_2 + \beta_3dy_3 \in \Omega^1(\R^6)  \, /  } \nonumber \\
& & \hspace{3mm} 2x_1\beta_1 + 2x_2\beta_2 + 2x_3\beta_3 = 0  \;\, \mathrm{and} \;\, \alpha_1x_1+\alpha_2x_2+\alpha_3x_3 -\beta_1y_1-\beta_2y_2-\beta_3y_3 =0 \}
\end{eqnarray*}
and
\begin{eqnarray*}
D_1^\circ & = & \{\alpha_1dx_1 + \alpha_2dx_2 + \alpha_3dx_3 + \beta_1dy_1 + \beta_2dy_2 + \beta_3dy_3 \in \Omega^1(\R^6) \,/  \nonumber \\
& &  \hspace{30mm} \alpha_1y_1+\alpha_2y_2+\alpha_3l_3y_3 - m_3\beta_3=0 \; \,\mathrm{and}\; \,\alpha_3 =0 \}.
\end{eqnarray*}
A basis of smooth sections of $D_0^\circ$ is the quadruple $(\zeta_1,\zeta_2,\zeta_3,\zeta_4)$ of smooth $1$-forms on $\R^6$ given by
\begin{equation*}
\zeta_1 = x_2dx_1 - x_1dx_2, \;\; \zeta_2 = x_3dx_2 - x_2dx_3, \;\;
\end{equation*}
\begin{equation*}
\zeta_3 = -y_2dx_1 + y_1dx_2+x_2dy_1 - x_1dy_2 \;\; \mathrm{and} \;\; \zeta_4 = -y_3dx_2 + y_2dx_3+x_3dy_2 - x_2dy_3.
\end{equation*}
The Pfaffian system $D_1^\circ$ is generated by the family $(\eta_1,\eta_2,\eta_3,\eta_4)$ of elements of $\Omega^1(\R^6)$, where
\begin{equation*}
\eta_1 = m_3dx_1+y_1dy_3, \;\; \eta_2 = m_3dx_2+y_2dy_3,\;\; \eta_3 = dy_1 \; \; \mathrm{and} \;\; \eta_4 = dy_2.
\end{equation*}
We consider the $2$-form
\begin{equation*}
\sigma_0 = \frac{y_2}{x_2^2}\zeta_1\wedge \zeta_2 + \frac{1}{x_2}\zeta_1\wedge \zeta_4 - \frac{1}{x_2}\zeta_2\wedge \zeta_3
\end{equation*}
on the open and dense subset $\mathcal{U}_0 = \{(x,y) \in \R^6 \, / \, x_2 \neq 0\}$ of $\R^6$. It is a smooth section of $\bigwedge^2 D_0^\circ$ of maximal rank on $\mathcal{U}_0$ and verifies the condition $\delta(\sigma_0 \wedge \sigma_0) = 2 \sigma_0 \wedge \delta(\sigma_0)$ (straightforward computation). So, its image $\Lambda^\#(\sigma_0)$ via $\Lambda^\#$ defines a linear Poisson structure $\Pi_0$ on $\mathcal{U}_0$ whose matricial expression is
\begin{equation*}
\Pi_0 = \left(
\begin{array}{cccccc}
0 & 0 & 0 & 0 & -x_3 & x_2 \\
0 & 0 & 0 & x_3 & 0 & -x_1 \\
0 & 0 & 0 & -x_2 & x_1 & 0 \\
0 & -x_3 & x_2 & 0 & -y_3& y_2 \\
x_3 & 0 & -x_1 & y_3 & 0 & -y_1 \\
-x_2 & x_1 & 0 & -y_2 & y_1 & 0
\end{array}
\right).
\end{equation*}
The center of the corresponding Lie algebra $(C^\infty(\mathcal{U}_0), \{\cdot,\cdot\}_0)$ is generated by the functions $f_1$ and $f_2$. We recognize $\Pi_0$ as the Lie-Poisson structure on the dual space of a Lie algebra isomorphic to $\mathfrak{e}(3) = \mathfrak{so}(3)\ltimes \R^3$. The vector field (\ref{Lagrange - top}) is Hamiltonian with respect to $\Pi_0$ and its Hamiltonian function is precisely the energy $f_3$ of the top. In the next step we search for a smooth section $\sigma_1 = \sum_{1\leq i <j \leq 4} k_{ij}\eta_i\wedge \eta_j$ of $\bigwedge^2 D_1^\circ$ of maximal rank on an open and dense subset $\mathcal{U}_1$ of $\R^6$ that verifies the Lenard-Magri recursion relations (\ref{recursion - sigma - X}), i.e.,
\begin{equation*}
\sigma_0 (X_{f_3}, \cdot) = \sigma_1 (X_{f_1}, \cdot) \quad \mathrm{and} \quad   \sigma_0 (X_{f_4}, \cdot) = \sigma_1 (X_{f_2}, \cdot).
\end{equation*}
After a long calculation we find that the coefficients $k_{ij}$ of such a $\sigma_1$ must be of the type
\begin{eqnarray*}
k_{12} & = & \frac{2k_{34}-l_3y_3}{2m_3x_3},  \nonumber \\
k_{13} & = & \frac{[2x_2(m_3x_2-y_2y_3)+2x_3y_2^2]k_{34}+ (m_3x_2-y_2y_3)(x_3y_2-l_3x_2y_3)-2x_2x_3y_2}{2m_3x_3(x_2y_1-x_1y_2)}, \nonumber \\
k_{14} & = & \frac{[2x_1(m_3x_2-y_2y_3)+2x_3y_1y_2]k_{34}+ (m_3x_2-y_2y_3)(x_3y_1-l_3x_1y_3)-2x_1x_3y_2}{2m_3x_3(x_1y_2-x_2y_1)}, \nonumber \\
k_{23} & = & \frac{[2x_2(m_3x_1-y_1y_3)+2x_3y_1y_2]k_{34}+ (m_3x_1-y_1y_3)(x_3y_2-l_3x_2y_3)-2x_2x_3y_1}{2m_3x_3(x_1y_2-x_2y_1)}, \nonumber \\
k_{24} & = & \frac{[2x_1(m_3x_1-y_1y_3)+2x_3y_1^2]k_{34}+ (m_3x_1-y_1y_3)(x_3y_1-l_3x_1y_3)-2x_1x_3y_1}{2m_3x_3(x_2y_1-x_1y_2)}, \nonumber \\
\end{eqnarray*}
with $k_{34}$ a some function on $\mathcal{U}_1 = \{(x,y)\in \R^6 \, / \, x_3(x_1y_2-x_2y_1) \neq 0 \}$. In the final step will must determine the function $k_{34}$ in a way such that $\sigma_1$ verifies also the conditions $\delta(\sigma_1 \wedge \sigma_1) = 2 \sigma_1 \wedge \delta(\sigma_1)$,
which is equivalent to saying that $\Pi_1 = \Lambda^\#(\sigma_1)$ is Poisson, and $\delta(\sigma_0 \wedge \sigma_1) =  \delta(\sigma_0) \wedge  \sigma_1 + \sigma_0 \wedge \delta(\sigma_1)$,
which assures us the compatibility of $\Pi_1$ with $\Pi_0$. These conditions are translated by a system of forty partial differential equations which is very difficult to solved. However, in the special case where $l_3=1 \Leftrightarrow I_3 =1$ and $m_3=2$, which means that the top is spherical and its center of gravity is the point $m=(0,0,2)$, by choosing $k_{34}=y_3/2$, we get $k_{12}=0$, $k_{13}=0$, $k_{14}=-1/2$, $k_{23}=1/2$ and $k_{24}=0$, which determine the $2$-form
\begin{equation*}
\sigma_1 = -\frac{1}{2}\eta_1\wedge \eta_4 +\frac{1}{2} \eta_2\wedge \eta_3 +\frac{y_3}{2}\eta_3\wedge\eta_4
\end{equation*}
of maximal rank on $\mathcal{U}_1$ verifying the required conditions. Therefore, the bivector field $\Pi_1=\Lambda^\#(\sigma_1)$ is Poisson, compatible with $\Pi_0$, and possesses as unique Casimir invariants the functions $f_3$ and $f_4$. Its matricial expression is
\begin{equation*}
\Pi_1 = \left(
\begin{array}{cccccc}
0 & \frac{1}{2}y_3 & -\frac{1}{2}y_2 & 0 & 1 & 0 \\
-\frac{1}{2}y_3 & 0 & \frac{1}{2}y_1 & -1 & 0 & 0 \\
\frac{1}{2}y_2 & -\frac{1}{2}y_1 & 0 & 0 & 0 & 0 \\
0 & 1 & 0 & 0 & 0 & 0 \\
-1 & 0 & 0 & 0 & 0 & 0 \\
0 & 0 & 0 & 0 & 0 & 0
\end{array}
\right).
\end{equation*}
Consequently, in this special case, $F^1$ and $F^2$ are the unique polynomial Casimir functions of the brackets $\{\cdot,\cdot\}_{(\lambda)}$ defined by the pencil
$\Pi_{(\lambda)} = \Pi_1 - \lambda \Pi_0$, $\lambda \in \R\cup\{\infty\}$, of affine Poisson structures on $\mathcal{U}=\mathcal{U}_0 \cap \mathcal{U}_1$. Thus, $\mathfrak{F}=(f_1,f_2,f_3,f_4)$, with $f_3 = \displaystyle{\frac{1}{2}}(y_1^2+y_2^2+y_3^2) + 2x_3$, is involutive with respect of any Poisson bracket $\{\cdot,\cdot\}_{(\lambda)}$, $\lambda \in \R\cup \{\infty\}$, which is given, in an alternative way, by the formula (\ref{bracket-lambda-even}). Precisely, we have $\Omega = \displaystyle{\frac{\omega^3}{3!}} = -dx_1\wedge dx_2 \wedge dx_3 \wedge dy_1 \wedge dy_2 \wedge dy_3$,
\begin{equation*}
F(\lambda) = \langle d(F^1)\wedge d(F^2), \Lambda \rangle = 2\lambda^2(x_1^2+x_2^2+x_3^2) - \lambda(y_1^2+y_2^2+y_3^2) + 2,
\end{equation*}
which is a polynomial function nonvanishing almost everywhere on $\R^6$,
\begin{eqnarray*}
\sigma_{(\lambda)} & = & \sigma_1 - \lambda \sigma_0  \nonumber \\
& = & \lambda y_3 dx_1\wedge dx_2 - \lambda y_2 dx_1\wedge dx_3 + \lambda y_1 dx_2\wedge dx_3 - (1+ \lambda x_3)dx_1\wedge dy_2 \nonumber \\
& & + \lambda x_2 dx_1\wedge dy_3 + (1+\lambda x_3)dx_2\wedge dy_1 - \lambda x_1 dx_2\wedge dy_3 - \lambda x_2 dx_3\wedge dy_1 \nonumber \\
& & + \lambda x_1 dx_3\wedge dy_2 + \frac{1}{2}y_3 dy_1\wedge dy_2 - \frac{1}{2}y_2 dy_1\wedge dy_3 + \frac{1}{2}y_1 dy_2\wedge dy_3,
\end{eqnarray*}
and $g_{(\lambda)} = i_{\Lambda} \sigma_{(\lambda)} = 0$. So,
\begin{eqnarray*}
\Phi_{(\lambda)} & = & -\frac{1}{F(\lambda)}\sigma_{(\lambda)}\wedge dF^1(\lambda) \wedge dF^2(\lambda) \nonumber \\
& = & - \lambda y_1 dx_1\wedge dx_2 \wedge dx_3 \wedge dy_1 - \lambda y_2 dx_1\wedge dx_2 \wedge dx_3 \wedge dy_2  \nonumber \\
& & - \lambda y_3 dx_1\wedge dx_2 \wedge dx_3 \wedge dy_3 - \lambda x_1 dx_1\wedge dx_2 \wedge dy_1 \wedge dy_3 \nonumber \\
& & - \lambda x_2 dx_1\wedge dx_2 \wedge dy_2 \wedge dy_3 + \lambda x_1 dx_1\wedge dx_3 \wedge dy_1 \wedge dy_2 \nonumber \\
& & - (1+\lambda x_3 )dx_1\wedge dx_3 \wedge dy_2 \wedge dy_3 - \frac{1}{2}y_1 dx_1\wedge dy_1\wedge dy_2\wedge dy_3 \nonumber \\
& & + \lambda x_2 dx_2\wedge dx_3 \wedge dy_1 \wedge dy_2 + (1+\lambda x_3 )dx_2\wedge dx_3 \wedge dy_1 \wedge dy_3  \nonumber \\
& & -\frac{1}{2}y_2 dx_2\wedge dy_1 \wedge dy_2 \wedge dy_3 - \frac{1}{2}y_3 dx_3\wedge dy_1 \wedge dy_2 \wedge dy_3
\end{eqnarray*}
and, for any $f,h\in C^\infty (\R^6)$,
\begin{equation*}
\{f,h \}_{(\lambda)}\Omega = df \wedge \, d h \wedge \Phi_{(\lambda)}.
\end{equation*}
Furthermore, we get that the vector field (\ref{Lagrange - top}) is also Hamiltonian with respect to $\Pi_1$ and its Hamiltonian is the function $f_1$. This bi-Hamiltonian formulation of (\ref{Lagrange - top}) was introduced in \cite{ratiu} on the semi-direct product $\mathfrak{so}(3)\ltimes \mathfrak{so}(3)$.
Another result is that the pencil $(\Pi_{(\lambda)})_{\lambda \in \R\cup\{\infty\}}$ possesses one more bi-Hamiltonian vector field, the vector field
\begin{equation*}
Y = x_2\frac{\partial}{\partial x_1} - x_1 \frac{\partial}{\partial x_2} + y_2\frac{\partial}{\partial y_1} - y_1 \frac{\partial}{\partial y_2},
\end{equation*}
whose Hamitlonians with respect to $\Pi_0$ and $\Pi_1$ are, respectively, the functions $f_4$ and $f_2$. By construction, $Y$ has the same constants of motion with (\ref{Lagrange - top}).

\subsection{From Lax pairs to Poisson pencils}
As it is well known and as we have noted in the introduction, any dynamical system $X$ on a smooth manifold $M$ which admits a Lax representation
\begin{equation*}
\dot{L}(t)=[L,P],
\end{equation*}
$(L,P)$ being the Lax pair of matrices, has as first integrals the coefficients of the chara\-cteristic polynomial $\det(L-\lambda I)$ of $L$. As an application of our construction of Poisson pencils with prescribed polynomial Casimir invariants, we can, starting from such a polynomial, derive Poisson pencils having $\det(L-\lambda I)$ as polynomial Casimir function and, consequently, bi-Hamiltonian systems possessing the same first integrals with $X$. If $M$ is of even dimension $2n$, in order to apply our technique, we cut
$\det(L-\lambda I)$ in an even number of parts $F^1(\lambda), \ldots, F^k(\lambda)$, $k=2l$, and construct Poisson pencils of rank $2r=2n-k$ almost everywhere on $M$
with $F^1(\lambda), \ldots, F^k(\lambda)$ as polynomial Casimir invariants. Consequently, $\det(L-\lambda I) = F^1(\lambda)+ \ldots + F^k(\lambda)$ is also a Casimir of the produced Poisson pencil.
Analogously, if $M$ is of odd dimension $2n+1$,
we cut $\det(L-\lambda I)$ in an odd number of parts $F^1(\lambda), \ldots, F^k(\lambda)$, $k=2l+1$, we add a dimension on $M$, and we consider also the polynomial $F^{k+1}(\lambda)=s$,
where $s$ is the coordinate on the added dimension. Then, we construct Poisson pencils on $M'=M\times \R$ of rank $2r= 2n +2-(k+1)=2n-2l$ almost anywhere
with polynomial Casimirs $F^1(\lambda),\ldots, F^{k+1}(\lambda)$. By reduction, we get Poisson pencils on $M$ of rank $2r$ almost anywhere with polynomial Casimir
functions $F^1(\lambda), \ldots, F^k(\lambda)$. Therefore, $\det(L-\lambda I) = F^1(\lambda)+ \ldots + F^k(\lambda)$ is such a polynomial also.

\vspace{2mm}
In the following, we use the algorithm outlined above to illustrate an explicit example of establishing Poisson pencils (of maximal rank)
with the characteristic polynomial of a Lax operator as polynomial Casimir. In order for the calculations to be reasonable, we consider the Lax formulation of a system on a manifold of low dimension.

Precisely, we deal with the system of classical, non-periodic, Toda lattice of $n=3$ particles \cite{joana-pan} that, in Flaschka's coordinate system $(a_1,a_2,b_1,b_2,b_3)$ \cite{flas}, takes the form
\begin{equation*}
\dot{a}_i = a_i(b_{i+1}-b_i),  \quad \dot{b}_i=2(a_i^2-a_{i-1}^2), \quad i=1,2,3,
\end{equation*}
with the convention $a_0 = a_3 =0$. These equations can be written in Lax form $\dot{L}=[L,P]$, where $L$ is the symmetric Jacobi matrix
\begin{equation*}
L = \left(
\begin{array}{ccc}
b_1 & a_1 & 0 \\
a_1 & b_2 & a_2 \\
0 & a_2 & b_3
\end{array}
\right)  \quad \quad \mathrm{and} \quad \quad  P = \left(
\begin{array}{ccc}
0 & -a_1 & 0 \\
a_1 & 0 & -a_2 \\
0 & a_2 & 0
\end{array}
\right).
\end{equation*}
The characteristic polynomial of $L$ is
\begin{equation*}
\det(L-\lambda I) = -\lambda^3 +(b_1+b_2+b_3)\lambda^2 +(a_1^2 +a_2^2 -b_1b_2 - b_1b_3 - b_2b_3)\lambda + b_1b_2b_3-b_1a_2^2-b_3a_1^2.
\end{equation*}
We set
\begin{eqnarray*}
f_0(a,b) & = & b_1+b_2+b_3   \nonumber \\
f_1(a,b) & = & a_1^2 +a_2^2 -b_1b_2 - b_1b_3 - b_2b_3 \nonumber \\
f_2(a,b) & = & b_1b_2b_3-b_1a_2^2-b_3a_1^2.
\end{eqnarray*}
Our aim is to build bi-Hamiltonian structures $(\Pi_0, \Pi_1)$ on $\R^5$, of maximal rank $4$, such that $f_0$ is a Casimir invariant of $\Pi_0$, $f_2$ is a Casimir of $\Pi_1$ and the Lenard-Magri recursion relations
\begin{equation*}
\Pi_0^\#(df_{k+1}) = \Pi_1^\#(df_k), \quad k=0,1,
\end{equation*}
are satisfied. Then, $\det(L-\lambda I)$ will be a polynomial Casimir of the Poisson pencil $\Pi_{(\lambda)} = \Pi_1-\lambda \Pi_0$, $\lambda \in \R\cup\{\infty\}$. Since the studied Toda system depends on $5$ variables, we add a dimension on $\R^5$ and denote its canonical coordinate by $s$. Then, we apply our procedure step by step in order to establish Poisson pencils $(\Pi_{(\lambda)})_{\lambda \in \R\cup\{\infty\}}$ on $\R^6$ admitting the polynomials $F^1(\lambda)=\det(L-\lambda I)$ and $F^2(\lambda)=s$ as Casimirs. (The function $s$ will be a common Casimir for all the Poisson structures of the pencil.) On $\R^6$, we take the canonical Poisson bivector field $\Lambda' = \displaystyle{\frac{\partial}{\partial a_1}\wedge \frac{\partial}{\partial b_1} + \frac{\partial}{\partial a_2}\wedge \frac{\partial}{\partial b_2} + \frac{\partial}{\partial s}\wedge \frac{\partial}{\partial b_3}}$ which verifies condition (\ref{condition-Lambda-0-1}), i.e.
\begin{equation*}
\langle df_0\wedge ds, \Lambda' \rangle = -1 \neq 0 \quad \mathrm{and} \quad \langle df_2\wedge ds, \Lambda' \rangle =-1\neq 0.
\end{equation*}
$\Lambda'$ can be viewed as the Poisson structure on $\R^6$ generated by the co-symplectic structure $(\vartheta, \Theta)= (db_3, da_1\wedge db_1 + da_2\wedge db_2)$ and its contravariant almost Jacobi structure $(\Lambda,E) = (\displaystyle{\frac{\partial}{\partial a_1}\wedge \frac{\partial}{\partial b_1} + \frac{\partial}{\partial a_2}\wedge \frac{\partial}{\partial b_2}, \frac{\partial}{\partial b_3}})$ on $\R^5$. The Hamiltonian vector fields of $f_i$, $i=0,1,2$, and $s$ with respect to $\Lambda'$ are:
\begin{equation*}
X'_{f_0}= -\frac{\partial}{\partial a_1} - \frac{\partial}{\partial a_2} - \frac{\partial}{\partial s}, \;\; X'_{f_1} = (b_2+b_3)\frac{\partial}{\partial a_1} + (b_1+b_3)\frac{\partial}{\partial a_2}+2a_1\frac{\partial}{\partial b_1} + 2a_2\frac{\partial}{\partial b_2} + (b_1+b_2)\frac{\partial}{\partial s},
\end{equation*}
\begin{equation*}
X'_{f_2} = (a_2^2-b_2b_3)\frac{\partial}{\partial a_1} - b_1b_3\frac{\partial}{\partial a_2} -2a_1b_3\frac{\partial}{\partial b_1} - 2a_2b_1\frac{\partial}{\partial b_2} + (a_1^2-b_1b_2)\frac{\partial}{\partial s} \quad \mathrm{and} \quad X'_s = \frac{\partial}{\partial b_3}.
\end{equation*}
So, $D'_0 = \langle X'_{f_0}, X'_s\rangle$ and $D'_1=\langle X'_{f_1},X'_s\rangle$, and their annihilators are, respectively, the subbundles of rank $4$ of $T^\ast\R^6$,
\begin{equation*}
D_0'^\circ = \{\alpha_1da_1 + \alpha_2da_2 + \beta_1db_1+\beta_2db_2+\beta_3db_3 +\gamma ds \in \Omega^1(\R^6) /\, \alpha_1+\alpha_2+\gamma =0 \;\mathrm{and}\; \beta_3=0\}
\end{equation*}
and
\begin{eqnarray*}
\lefteqn{D_1'^\circ = \{\alpha_1da_1 + \alpha_2da_2 + \beta_1db_1+\beta_2db_2+\beta_3db_3 +\gamma ds \in \Omega^1(\R^6) /  } \nonumber \\
& & \hspace{8mm} (a_2^2-b_2b_3)\alpha_1-b_1b_3 \alpha_2 +(a_1^2-b_1b_2)\gamma -2a_1b_3\beta_1-2a_2b_1\beta_2=0 \;\mathrm{and}\; \beta_3=0\}.
\end{eqnarray*}
A basis of smooth sections of $D_0'^\circ$ is the collection $(\zeta_1, \zeta_2,\zeta_3,\zeta_4)$ of smooth $1$-forms on $\R^6$, written as
\begin{equation*}
\zeta_1 = da_1 - da_2, \quad \zeta_2 = da_2-ds, \quad \zeta_3 = db_1 \quad \mathrm{and} \quad \zeta_4 = db_2,
\end{equation*}
while a basis of smooth sections of $D_1'^\circ$ is the family $(\eta_1, \eta_2, \eta_3, \eta_4)$ of elements of $\Omega^1(\R^6)$ given by the formul{\ae}
\begin{equation*}
\eta_1 = -2b_2da_2+a_1db_1 +2b_3ds, \quad \eta_2=2b_1da_1-2b_2da_2 +a_2db_2,
\end{equation*}
\begin{equation*}
\eta_3 = 2a_1da_2-b_1db_1 \quad \mathrm{and} \quad \eta_4=2a_2da_2-b_3db_2.
\end{equation*}

\vspace{2mm}
\noindent
\emph{First selection:} We choose on the open and dense subset $\mathcal{U}'=\{(a_1,a_2,b_1,b_2,b_3,s)\in \R^6\, / \, b_1b_3 \neq 0\}$ of $(\R^6, \Lambda')$ the $2$-forms
\begin{equation*}
\sigma_0' = -a_1\zeta_1\wedge \zeta_3 - a_2 \zeta_2\wedge \zeta_4  \quad \mathrm{and} \quad \sigma_1' = -\frac{a_2}{2b_3}\eta_1\wedge \eta_4 + \frac{a_1}{2b_1}\eta_2\wedge \eta_3 -\frac{a_1a_2}{2b_1b_3}\eta_3\wedge \eta_4,
\end{equation*}
which are smooth sections of maximal rank on $\mathcal{U}'$ of $\bigwedge^2 D_0'^\circ$ and $\bigwedge^2 D_1'^\circ$, respectively. By a long, but straightforward, calculation, we confirm that the pair $(\sigma_0', \sigma_1')$ verifies conditions (\ref{sigma}) and (\ref{recursion - sigma - X}). Hence, according to Theorem \ref{central-theorem}, the image by $\Lambda'\,^\#$ of the $1$-parameter family $\sigma'_{(\lambda)}= \sigma_1' - \lambda \sigma_0'$, $\lambda \in \R\cup\{\infty\}$, of $2$-forms of maximal rank on $\mathcal{U}'$ defines a Poisson pencil $(\Pi'_{(\lambda)})_{\lambda \in \R\cup\{\infty\}}$ on $\mathcal{U}'$ with the required properties. Since $s$ is a Casimir of this pencil, whose the elements are independent of $s$ and without a term of type $\cdot \wedge \frac{\partial}{\partial s}$, it can be viewed as a Poisson pencil on the open and dense subset $\mathcal{U}$ of $\R^5$ which is the projection of $\mathcal{U}'$ on the hyperplane $s=0$ of $\R^6$. The reduced from $(\Pi'_0,\Pi'_1)$ bi-Hamiltonian structure $(\Pi_0,\Pi_1)$ on $\mathcal{U}$ is the well known pair of a linear and a quadratic Poisson tensors with respect to which the classical, non-periodic, Toda lattice of $3$ particles is bi-Hamiltonian:
\begin{equation*}
\Pi_0 \:=\: \left(
\begin{array}{ccccc}
0 & 0 & -a_1 &a_1 & 0 \\
0 & 0 & 0 &-a_2 & a_2 \\
a_1 & 0 & 0 & 0 & 0\\
-a_1 & a_2 & 0 & 0 & 0\\
0& -a_2 & 0 &0 & 0
\end{array}
\right),
\end{equation*}
\begin{equation*}
\Pi_1 \:=\: \left(
\begin{array}{ccccc}
0 & \frac{1}{2}a_1a_2 & -a_1b_1 &a_1b_2 & 0 \\
-\frac{1}{2}a_1a_2 & 0 & 0 &-a_2b_2 & a_2b_3 \\
a_1b_1 & 0 & 0 & 2a_1^2 & 0\\
-a_1b_2 & a_2b_2 & -2a_1^2 & 0 & 2a_2^2 \\
0& -a_2b_3 & 0 &-2a_2^2 & 0
\end{array}
\right).
\end{equation*}
Under Flaschka's transformation, the linear structure $\Pi_0$ is the image of the standard symplectic bracket on $\R^6$ and the quadratic one was founded by Adler in \cite{adler}. The Hamiltonian vector fields generated by $(\Pi_0,f_1)$ and $(\Pi_1,f_0)$ are equal and coincide with the Toda system, whence we get that the Toda system is bi-Hamiltonian. The Poisson pencil $(\Pi_{(\lambda)})_{\lambda \in \R\cup \{\infty\}}$, $\Pi_{(\lambda)} = \Pi_1 - \lambda \Pi_0$, possesses also the bi-Hmailtonian vector field $\Pi_0^\#(df_2)=\Pi_1^\#(df_1)$. The function $f_0$ is the only Casimir function of $\Pi_0$, $f_2$ is the only Casimir of $\Pi_1$, and $\det(L-\lambda I)+\lambda^3$ is the only Casimir of the pencil $(\Pi_{(\lambda)})_{\lambda \in \R\cup \{\infty\}}$. Hence, $\mathfrak{F}=(f_0,f_1,f_2)$ is involutive with respect to any Poisson bracket $\{\cdot,\cdot\}_{(\lambda)}$, $\lambda \in \R\cup \{\infty\}$. The latter is determined, in an equivalent way with the original, by formula (\ref{bracket-lambda-odd}). We have $\Omega = \vartheta\wedge \Theta^2 = - da_1\wedge da_2\wedge db_1\wedge db_2 \wedge db_3$, and
\begin{equation*}
F(\lambda) = \langle dF^1(\lambda), \, E  \rangle = \langle df_0 \lambda^2 + df_1\lambda +df_2, \frac{\partial}{\partial b_3} \rangle = \lambda^2 - (b_1+b_2)\lambda + b_1b_2-a_1^2,
\end{equation*}
which is a polynomial function, nonvanishing almost everywhere on $\R^5$. The semi-basic part $\sigma_{(\lambda)}=\sigma_1 - \lambda \sigma_0$ of $\sigma'_{(\lambda)}=\sigma'_1 - \lambda \sigma'_0$ is
\begin{eqnarray*}
\sigma_{(\lambda)} & = & \sigma_1 - \lambda \sigma_0 \nonumber \\
& =& 2a_1^2da_1\wedge da_2 +(a_1\lambda - a_1b_1)da_1\wedge db_1 + (a_1b_2-a_1\lambda)da_2\wedge db_1 \nonumber \\
& & +\,(a_2\lambda -a_2b_2)da_2\wedge db_2 + \frac{1}{2}a_1a_2db_1\wedge db_2
\end{eqnarray*}
and $g_{(\lambda)} = i_{\Lambda}\sigma_{(\lambda)} = a_1b_1+a_2b_2-(a_1+a_2)\lambda$. By replacing the above expressions in (\ref{bracket-lambda-odd}), we get
\begin{eqnarray*}
\Phi_{(\lambda)} & = & -\frac{1}{F(\lambda)}(\sigma_{(\lambda)} + g_{(\lambda)}\Theta)\wedge dF^1(\lambda) \nonumber \\
& = & -\,2a_2^2da_1\wedge da_2\wedge db_1 -2a_1^2 da_1\wedge da_2\wedge db_3 +(\lambda a_2-a_2b_3)da_1\wedge db_1\wedge db_2 \nonumber \\
& & +\,(\lambda a_2-a_2b_2)da_1\wedge db_1\wedge db_3 +(\lambda a_1 - a_1b_2)da_2\wedge db_1\wedge db_3  \nonumber \\
& & +\,(\lambda a_1 -a_1b_1)da_2\wedge db_2\wedge db_3 -\frac{1}{2}a_1a_2db_1\wedge db_2\wedge db_3
\end{eqnarray*}
and, for any $f,h \in C^\infty(\R^5)$,
\begin{equation*}
\{f,h\}_{(\lambda)}\Omega = df\wedge dh\wedge \Phi_{(\lambda)}.
\end{equation*}

\vspace{2mm}
\noindent
\emph{Second selection:} We choose another pair $(\sigma_0',\sigma_1')$ of $2$-forms on $(\R^6, \Lambda')$ verifying the assumptions of Theorem \ref{central-theorem}. Let $\sigma_0'$ and $\sigma_1'$ be the sections of maximal rank of $\bigwedge^2 D_0'^\circ$ and $\bigwedge^2 D_1'^\circ$, respectively, given by the relations
\begin{equation*}
\sigma_0' = 2a_1b_1\zeta_1\wedge \zeta_3 +2a_2b_3\zeta_2\wedge\zeta_4 - a_1a_2\zeta_3\wedge \zeta_4 \quad \mathrm{and} \quad \sigma_1'= a_2\eta_1\wedge \eta_4 - a_1\eta_2\wedge \eta_3.
\end{equation*}
After a very long, but straightforward, computation, we show that $(\sigma_0', \sigma_1')$ satisfies conditions (\ref{sigma}) and (\ref{recursion - sigma - X}). Thus, the image via $\Lambda'^\#$ of the $1$-parameter family $\sigma_{(\lambda)}'=\sigma_1'-\lambda\sigma_0'$, $\lambda \in \R\cup\{\infty\}$, of $2$-forms of rank $4$ on $\R^6$ defines another Poisson pencil $(P'_{(\lambda)})_{\lambda \in \R\cup\{\infty\}}$ on $\R^6$ with the required properties. As in the previous example and for the same reasons, the induced Poisson pencil $(P_{(\lambda)})_{\lambda \in \R\cup\{\infty\}}$ from $(P'_{(\lambda)})_{\lambda \in \R\cup\{\infty\}}$ on the submanifold $\R^5$ of $\R^6$ defined by the equation $s=0$, is a Poisson pencil of maximal rank having as Casimir invariant the polynomial $\det(L-\lambda I)$. The components of the pair $(P_0,P_1)$ of bivector fields on $\R^5$ induced by $(P_0',P_1')$, $P_0'=\Lambda'^\#(\sigma_0')$ and $P_1'=\Lambda'^\#(\sigma_1')$, have, respectively, the matricial expressions
\begin{equation*}
P_0 \:=\: \left(
\begin{array}{ccccc}
0 & -a_1a_2 & 2a_1b_1 &-2a_1b_1 & 0 \\
a_1a_2 & 0 & 0 &2a_2b_3 & -2a_2b_3 \\
-2a_1b_1 & 0 & 0 & 0 & 0\\
2a_1b_1 & -2a_2b_3 & 0 & 0 & 0\\
0& 2a_2b_3 & 0 &0 & 0
\end{array}
\right)
\end{equation*}
and
\begin{equation*}
P_1 \:=\: \left(
\begin{array}{ccccc}
0 & -a_1a_2(b_1+b_3) & 2a_1b_1^2 &-2a_1(a_2^2+b_1b_2) & 0 \\
a_1a_2(b_1+b_3) & 0 & 0 &2a_2(a_1^2+b_2b_3) & -2a_2b_3^2 \\
-2a_1b_1^2 & 0 & 0 & -4a_1^2b_1 & 0\\
2a_1(a_2^2+b_1b_2)& -2a_2(a_1^2+b_2b_3) & 4a_1^2b_1 & 0 & -4a_2^2b_3\\
0& 2a_2b_3^2 & 0 &4a_2^2b_3 & 0
\end{array}
\right).
\end{equation*}
By construction, $\mathfrak{F}$ is also involutive with respect to each Poisson bracket $\{\cdot,\cdot\}_{(\lambda)}$ on $C^\infty(\R^5)$ defined by $P_{(\lambda)}$, $\lambda \in \R\cup \{\infty\}$, which can be expressed by (\ref{bracket-lambda-odd}), and the bi-Hamiltonian vector fields $P_0^\#(df_1) = P_1^\#(df_0)$ and $P_0^\#(df_2) = P_1^\#(df_1)$ on $\R^5$ have the same first integrals with the Toda lattice.

It is interesting to note that $\Pi_0$ and $P_0$ are compatible Poisson tensors having the same Casimir function, the function $f_0$, $\Pi_1$ is compatible with $P_1$ and they also have the same Casimir invariant, the function $f_2$, while the pairs $(\Pi_0,P_1)$ and $(\Pi_1,P_0)$ are not compatible.

\end{document}